\documentclass[a4paper,11pt]{amsart}
\usepackage{amssymb}
\usepackage{amscd}
\usepackage{comment}
\usepackage{amsmath,amsthm}
\usepackage[colorlinks=true]{hyperref}
\usepackage{enumerate}
\usepackage{booktabs,multirow}
\usepackage{tikz}
\usepackage{rotating}
\usetikzlibrary{patterns}
\usetikzlibrary{decorations.pathreplacing}
\usetikzlibrary{calc,through}

\allowdisplaybreaks[1]
\numberwithin{figure}{section}
\numberwithin{equation}{section}

\title{Rational Dyck paths and decompositions}
\author[K.~Shigechi]{Keiichi~Shigechi}
\email{k1.shigechi AT gmail.com}
\date{\today}

\newcommand\tikzpic[2]{
\raisebox{#1\totalheight}{
\begin{tikzpicture}
#2
\end{tikzpicture}
}}
\newcommand\circnum[1]{
\hspace{-3mm}
\tikzpic{-0.22}{[scale=0.6]
\node[draw,circle,inner sep=1pt]at(0,0){$#1$};
}
\hspace*{-3mm}
}
\newcommand\cnum[1]{
\tikzpic{}{[scale=0.7]
\node[draw,circle,inner sep=2pt]at(0,0){$#1$};
}
}
\newtheorem{theorem}[figure]{Theorem}
\newtheorem{example}[figure]{Example}
\newtheorem{lemma}[figure]{Lemma}
\newtheorem{defn}[figure]{Definition}
\newtheorem{prop}[figure]{Proposition}
\newtheorem{cor}[figure]{Corollary}

\newtheorem{remark}[figure]{Remark}
\begin{document}
\begin{abstract}
We study combinatorial properties of a rational Dyck path by decomposing 
it into a tuple of Dyck paths.
The combinatorial models such as $b$-Stirling permutations, $(b+1)$-ary 
trees, parenthesis presentations, and binary trees play central roles 
to establish a correspondence between the rational Dyck path and the 
tuple of Dyck paths.
We reinterpret two orders, the Young and the rotation orders, on 
rational Dyck paths in terms of the tuple of Dyck paths by use of 
the decomposition.
As an application, we show a duality between $(a,b)$-Dyck paths 
and $(b,a)$-Dyck paths in terms of binary trees. 
\end{abstract}

\maketitle

\section{Introduction}
In this article, we study combinatorics on a rational Dyck path by the decompositions 
of it into a tuple of Dyck paths.
A rational Dyck path, called an $(a,b)$-Dyck path, of size $n$ is a lattice path 
with a certain condition explained in Section \ref{sec:rDp}.
We consider several combinatorial models associated to a rational Dyck path such 
as $b$-Stirling permutations, $(b+1)$-ary trees, parenthesis presentations, the horizontal
and vertical strip decompositions, and binary trees.

We consider two operations on an $(a,b)$-Dyck path, called strip decompositions: one is 
the horizontal strip decomposition, and the other is the vertical strip decomposition.
They are generalizations of the strip decomposition of a $(1,b)$-Dyck path 
studied in \cite{KalMuh15}.
These two decompositions produce a tuple of Dyck paths.
The advantage of the decompositions is that we translate the properties of 
an $(a,b)$-Dyck path into the properties of a tuple of $(1,1)$-Dyck paths, 
and study them on $(1,1)$-Dyck paths.
For example, we have two distinguished orders, the Young order and the rotation order, 
on the poset of rational Dyck paths.
One of such translations is to interpret the rotation order on a tuple of Dyck paths.
In this case, the order between two $(a,b)$-Dyck paths can be reduced to a 
simple condition on $(1,1)$-Dyck paths.
We also study these Young and rotation orders on $(b+1)$-ary trees, or 
equivalently $b$-Stirling permutations \cite{GesSta78} 
(see also \cite{CebGonDLe19,Par94a,Par94b} and references therein).
Some remarks on the orders are in order.
The $b$-Dyck paths and the $b$-Tamari lattices are studied in \cite{KalMuh15}.
In \cite{PreRatVie17}, they give an extension of classical Tamari lattice \cite{FriTam67,Tam62} 
and $b$-Tamari lattice \cite{BerPreRat12}, and study its relation to binary trees.
There, the order is realized in terms of rational Dyck paths.
A generalizations of Tamari lattice is studied as $\nu$-Tamari 
lattice in \cite{CebPadSar18,CebPadSar19}.
In \cite{BGDMCY21}, they give a unifying framework for the Young and rotation orders.

We also study two parenthesis presentations of a rational Dyck path $P$.
We call them type $I$ and type $II$.
A parenthesis presentation of both types is obtained from a $b$-Stirling permutation, or
equivalently, a $(b+1)$-ary tree corresponding to $P$.
Since a rational Dyck path naturally defines a partition, we pass through 
the $312$-avoiding Stirling permutations to clarify the correspondence between 
a parenthesis presentation and a rational Dyck path.
A rational Dyck path has a presentation by the word consisting $N$  
and $E$, where $N$ (resp. $E$) corresponds to an up (resp. down) step.
For a Dyck path $P_1$ with $(a,b)=(1,1)$, we have the parenthesis presentation
obtained from $P_1$ by replacing a $N$ by the left parenthesis ``$($" and 
an $E$ by the right parenthesis ``$)$".
The two parenthesis presentations as mentioned above are natural generalizations 
of the parenthesis presentation for a $(1,1)$-Dyck path.
To capture the structure of a parenthesis presentation, we have 
$b$ $(1,1)$-Dyck paths.

The first parenthesis presentation, which is of type $I$, behaves nicely 
for the rotation of $P$, and the second one behaves nicely when we 
consider binary trees. 
Actually, we translate the rotation of $P$ in the language of a $b$-tuple of Dyck paths  
obtained by the first parenthesis presentation.
To obtain a natural correspondence between a rational Dyck path and 
a parenthesis presentation, $312$-avoiding $b$-Stirling permutations 
play a central role. This is because the parenthesis presentations 
capture the structure of a $(b+1)$-ary tree, however, do not 
capture the information of labels on edges in the $(b+1)$-ary tree.

Given a lowest $(a,b)$-Dyck path $P_0$, one can consider a binary tree
associated to $P_0$.
One can define a way of giving a word by visiting the edges of the 
binary tree.
The first way gives the path $P_0$ itself. 
By introducing another way of obtaining a word from the binary tree, 
we encode two $(a,b)$-Dyck paths $P_0$ and $Q$ as the reading words 
obtained from the tree.
In \cite{PreRatVie17}, they introduce a new algorithm called ``push-gliding" 
which is the algorithm to obtain a binary tree from two Dyck paths.
This algorithm transforms two Dyck paths into a binary tree which 
contains two reading words for $P_0$ and $Q$. 
In this paper, we give another algorithm to produce a binary tree 
for two Dyck paths.
The algorithm is starting from a binary tree for $P_{0}$, then 
cut it into several pieces and make a new binary tree for $P_0$ and 
$Q$ by reconnecting them into a single tree. 
We also show that the parenthesis presentation of type $II$ is compatible with 
the Dyck paths obtained from the binary tree for $P_0$ and $Q$ 
by dividing left edges into $b$ pieces and by reading the words.	
We study the relation between an $(a,b)$-Dyck path $P$ and the horizontal and vertical
strip decompositions.
For this, we consider a $(b,a)$-Dyck path $P^{\sharp}$ which is regarded as a dual path of $P$. 
By a strip decomposition, we obtain $a$ (in case of $a<b$) binary trees for both 
$P$ and $P^{\sharp}$.
Since $P$ and $P^{\sharp}$ are ``dual" to each other, 
the binary trees for $P$ obtained by the horizontal strip decomposition are dual to 
those for $P^{\sharp}$ obtained from the parenthesis presentation of type $II$.

This paper is organized as follows.
In Section \ref{sec:rDp}, we introduce rational Dyck paths and summarize 
their basic properties such as step and height sequences, a horizontal and 
vertical strip decompositions, and the Young and rotation orders.
In Section \ref{sec:multip}, we study $b$-Stirling permutations and 
$(b+1)$-ary trees. 
The Young and rotation orders are also translated in the language of 
$(b+1)$-ary trees.
In Section \ref{sec:312}, the parenthesis presentation is introduced 
to obtain multiple Dyck paths.
The rotation on a rational Dyck path is also translated into the 
rotation on multiple Dyck paths.
In Sections \ref{sec:sdpp} and \ref{sec:Hsdmp}, 
we continue to study parenthesis presentations and multiple Dyck 
paths.  
In Section \ref{sec:bt}, we connect the results in previous sections 
to binary trees.

\section{Rational Dyck paths}
\label{sec:rDp}
\subsection{Rational Dyck paths}
Let $(a,b)\in\mathbb{N}^{2}$ be relatively prime positive integers 
and $n\in\mathbb{N}$.
A lattice path from $(0,0)$ to $(bn,an)$ staying above the line $y=ax/b$
is called $(a,b)$-Dyck path, or called {\it rational Dyck path} \cite{Biz54}.
We call the integer $n$ the size of an $(a,b)$-Dyck path.
We denote by $\mathfrak{D}^{(a,b)}_{n}$ the set of $(a,b)$-Dyck path 
of size $n$.
A $(a,b)$-Dyck path is a generalization of the well-known Dyck paths ($(a,b)=(1,1)$) 
and $b$-Dyck paths ($a=1$).

The cardinality of $\mathfrak{D}_{n}^{(1,b)}$ is given by the Fuss-Catalan number
\begin{align*}
\mathrm{Cat}^{(b)}_{n}=\genfrac{}{}{0.8pt}{}{1}{bn+1}
\begin{pmatrix}(b+1)n \\ n\end{pmatrix},
\end{align*}
and specialization of $a=1$ yields the well-studied Catalan numbers.

Following \cite{KalMuh15}, we introduce weakly increasing sequences, the step 
and the height sequences, for a rational Dyck path. 
Let $P$ be a $(a,b)$-Dyck path of size $n$.
The {\it step sequence} $\mathfrak{u}_{P}:=(u_1,u_2,\ldots,u_{an})$ 
is a sequence of non-negative integers defined for $P$ as 
\begin{align*}
&u_{1}\le u_{2}\le\ldots\le u_{an},\\
&u_{k}\le \genfrac{}{}{0.8pt}{}{b}{a}(k-1),\quad \forall k\in[1,an].
\end{align*}
The entry $u_{k}$ in the height sequence indicates that 
the path $P$ passes through the edge connecting 
$(u_{k},k-1)$ and $(u_{k},k)$.
Similarly, the {\it height sequence} $\mathfrak{h}_{P}=(h_1,h_2,\ldots,h_{bn})$
is a sequence of positive integers satisfying 
\begin{align*}
&h_{1}\le h_2\le \ldots \le h_{bn}, \\
&h_{k}\ge \lceil ka/b \rceil, \quad \forall k\in[1,bn].
\end{align*}
The entry $h_{k}$ in the height sequence indicates that 
the path $P$ passes through the edge connecting 
$(k-1,h_{k})$ and $(k,h_{k})$.
Note that one can yield a $(a,b)$-Dyck path of size $n$ 
by either the step sequence or the height sequence.

Since a rational path $P$ in $\mathfrak{D}_{n}^{(a,b)}$ is a lattice path from $(0,0)$
to $(bn,an)$, each step is either $(0,1)$ or $(1,0)$. 
When a step is $(0,1)$, we write $N$ (a north step). Similarly, when a step is 
$(1,0)$, we write $E$ (an east step).
We have an expression of $P$ by a sequence of $N$'s and $E$'s, and denote 
it by $P$ by abuse of notation.
The path $P$ has $an$ $N$'s and $bn$ $E$'s.

Figure \ref{fig:rDyck} gives an example of a rational Dyck path $P=NENENEENEE$.
\begin{figure}[ht]
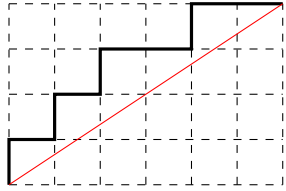

\tikzpic{-0.5}{[scale=0.6]
\draw[very thick](0,0)--(0,1)--(1,1)--(1,2)--(2,2)--(2,3)--(4,3)--(4,4)--(6,4);
\draw[dashed](0,0)--(0,4)(1,0)--(1,4)(2,0)--(2,4)(3,0)--(3,4)(4,0)--(4,4)(5,0)--(5,4)(6,0)--(6,4);
\draw[dashed](0,0)--(6,0)(0,1)--(6,1)(0,2)--(6,2)(0,3)--(6,3)(0,4)--(6,4);
\draw[red](0,0)--(6,4);
}
\caption{A rational Dyck path $P$ in $\mathfrak{D}_{2}^{(2,3)}$ with step sequence
$\mathfrak{u}_{P}=(0,1,2,4)$ and height sequence $\mathfrak{h}_{P}=(1,2,3,3,4,4)$.}
\label{fig:rDyck}
\end{figure}

Let $P_{0}\in\mathfrak{D}_{n}^{(a,b)}$ be the lowest Dyck path in $\mathfrak{D}_{n}^{(a,b)}$ 
associated with $\mathfrak{u}_{P_{0}}:=(u_1,\ldots,u_{an})$.
In other words, there exists no $P'$ associated with $\mathfrak{u}_{P'}:=(u'_1,\ldots,u'_{an})$ 
such that 
\begin{align}
&u_{k}=u'_{k}, \quad \forall k\in[1,i],\\
&u_{i+1}<u'_{i+1},
\end{align}
for some $i\in[1,an]$.

Similarly, let $P_{1}\in\mathfrak{D}_{n}^{(a,b)}$ be the highest path in $\mathfrak{D}_{n}^{(a,b)}$ 
associated with $\mathfrak{u}_{P_{1}}:=(u_1,\ldots,u_{an})$.
In other words, there exists no $P'$ associated with $\mathfrak{u}_{P'}:=(u'_1,\ldots, u'_{an})$ 
such that 
\begin{align}
&u_{k}=u'_{k}, \quad \forall k\in[1,i],\\
&u_{i+1}>u'_{i+1},
\end{align}
for some $i\in[1,an]$.

For example, the lowest and highest paths $P_0$ and $P_1$ in $\mathfrak{D}_{2}^{(2,3)}$
are given by $P_{0}=NENE^{2}NENE^{2}$ and $P_1=N^{4}E^{6}$.

\subsection{Rotation order}
\label{sec:rot}
We generalize the notion of the rotation order studied in \cite{KalMuh15} to 
the rotation on $(a,b)$-Dyck paths.

Let $P\in\mathfrak{D}_{n}^{(a,b)}$ with step sequence 
$\mathfrak{u}_{P}=(u_1,u_2,\ldots,u_{an})$.
Following \cite{BerPreRat12}, for $i\in[1,an]$, we say that the {\it primitive subsequence of $\mathfrak{u}_{P}$
at position $i$}  is the unique subsequence 
$(u_{i},u_{i+1},\ldots,u_{k})$ 
such that it satisfies the following two conditions:
\begin{align}
\label{eqn:condrot}
\begin{aligned}
&u_{j}-u_{i}<\frac{b}{a}(j-i),\quad \forall j\in[i+1,k],  \\
&\text{either } k=n, \text{ or } u_{k+1}-u_{i}\ge\frac{b}{a}(k+1-i). 
\end{aligned}
\end{align}

Let $P,P'\in\mathfrak{D}_{n}^{(a,b)}$ be rational Dyck paths 
such that the step sequences of $P$ and $P'$
are $\mathfrak{u}_{P}$ and $\mathfrak{u}_{P'}$, respectively.

We define a cover relation $\lessdot_{\mathrm rot}$ on $P$ and $P'$ by 
\begin{align}
\label{eqn:defrot}
P\lessdot_{\mathrm rot}P' \Leftrightarrow 
\mathfrak{u}_{P'}=(u_1,\ldots,u_{i-1},u_{i}-1,\ldots,u_{k}-1,u_{k+1},\ldots,u_{an}),
\end{align}
for some $i\in[2,an]$ with $u_{i-1}<u_{i}$ such that $\{u_i,u_{i+1},\ldots,u_{k}\}$ 
is the primitive subsequence of $\mathfrak{u}_{P}$ at position $i$.
The {\it rotation order} on $\mathfrak{D}_{n}^{(a,b)}$ is defined from the cover 
relation $\lessdot_{\mathrm rot}$ by the transitive and reflexive closure of it.

\begin{example}
We consider the same rational Dyck path $P$ in Figure \ref{fig:rDyck}.
The step sequence for $P$ is $(0,1,2,4)$.
When $P\lessdot_{\mathrm{rot}}P'$, the step sequence for $P'$ 
is either $(0,0,1,4)$, $(0,1,1,4)$, or $(0,1,2,3)$.
\end{example}

In \cite{PreRatVie17}, they give an alternative definition of the rotation order by using 
the lowest path $P_0\in\mathfrak{D}_{n}^{(a,b)}$.
Let $P$ be a rational Dyck path above $P_0$ and $\mathfrak{u}_{0}:=(u_{0}(1),\ldots,u_0(an))$ 
be the step sequence of $P_{0}$.
Let $p=(p_{x},p_{y})$ be a lattice point on $P$. 
The horizontal distance $\mathrm{hor}_{P_{0}}(p)$ is defined by 
\begin{align*}
\mathrm{hor}_{P_{0}}(p):=u_{0}(p_{y}+1)-p_{x}.
\end{align*}
Suppose that $p$ is a lattice point such that it preceded by a step $E$ and 
followed by a step $N$.
Let $p'$ be the first lattice point after $p$ such that 
$\mathrm{hor}_{P_0}(p')=\mathrm{hor}_{P_0}(p)$.
Let $P[p,p']$ be the subpath from $p$ to $p'$ in $P$.
Let $P'$ be a path obtained from $P$ by switching the $E$ step at $p$ and $P[p,p']$.
We define the covering relation to be $P\lessdot_{P_0}P'$.

\begin{example}
The step sequence of the lowest path $P_0$ in $\mathfrak{D}_{2}^{(2,3)}$ 
is given by $\mathfrak{u}_{P_0}=(0,1,3,4)$.
We consider the Dyck path $P$ in Figure \ref{fig:rDyck}, whose 
step sequence is $(0,1,2,4)$.
We consider the lattice point $p_1=(1,1)$.
The horizontal distance $\mathrm{hor}_{P_0}(p_1)=\mathfrak{u}_{P_0}(2)-1=0$ 
is equal to the horizontal distance at $p_2=(4,3)$.
Then, the subpath is $P[p_1,p_2]=NENEE$ and we obtain a new path $P'$ whose 
step sequence is $(0,0,1,4)$.
\end{example}

We show that the rotation orders studied in \cite{KalMuh15} and \cite{PreRatVie17}
are equivalent in case of $a<b$.

\begin{remark}
It is easy to see that the two orders are different in case of $a>b$.
Let $P$ be the rational Dyck path in $\mathfrak{D}_{3}^{(2,1)}$ 
with the step sequence $(0,0,1,1,2,2)$.
Note that the path $P$ is the lowest path.
We consider the rotation at position $3$.
Then, we have $P\lessdot_{rot}P_1$, where the step sequence of $P_1$ 
is $(0,0,0,0,1,1)$.
On the other hand, we have $P\lessdot_{P}P_2$, where the step 
sequence of $P_2$ is $(0,0,0,1,2,2)$.
\end{remark}

\begin{prop}
\label{prop:rot2}
Let $P_0$ be the lowest path in $\mathfrak{D}_{n}^{(a,b)}$.
Two covering relations  $\lessdot_{rot}$ and $\lessdot_{P_{0}}$ give 
the same covering relation if $a<b$.
\end{prop}
\begin{proof}
The horizontal distance $\mathrm{hor}_{P_0}(p)$ is a distance 
from the lattice point $p$ in $P$ to the right-most lattice point
in the same row above the line $y=ax/b$ since the lowest path 
$P_0$ is the right-most path above the line $y=ax/b$.
On the other hand, the condition (\ref{eqn:condrot}) for
the primitive subsequence of $\mathfrak{u}_{P}$ at $i$ 
detects the first lattice point at $j>i$ such that 
the horizontal distance at $j$ is equal to the horizontal 
distance at $i$.
The conditions that $u_{i-1}<u_{i}$ and $i\in[2,an]$ to define 
the rotation order on $\mathfrak{D}_{n}^{(a,b)}$ is equivalent 
to the condition that the lattice point $p$ is preceded by 
a step $E$ and followed by a step $N$.

We decrease $u_j$ by one in a cover relation (\ref{eqn:defrot}),
which is nothing but switching the $E$ path at $p$ and 
subpath $P[p,p']$.
From these observations, the two cover relations 
$\lessdot_{\mathrm{rot}}$ and $\lessdot_{P_{0}}$
give the same cover relation.
\end{proof}

\subsection{Young order}
Let $P\in\mathfrak{D}_{n}^{(a,b)}$ associated with the step sequence 
$\mathfrak{u}_{P}:=(u_{1},\ldots,u_{an})$.
Since the entries $u_{i}$ in $\mathfrak{u}_{P}$ are in an increasing 
order, we have a natural bijection between $\mathfrak{u}_{P}$ and 
a Young diagram with $an$ rows.

Let $P,P'\in\mathfrak{D}_{n}^{(a,b)}$ such that the step sequences of $P$ and $P'$
are $\mathfrak{u}_{P}$ and $\mathfrak{u}_{P'}$, respectively.

We define a cover relation $\lessdot_{\mathrm Y}$ on $P$ and $P'$ by 
\begin{align}
P\lessdot_{\mathrm Y}P' \Leftrightarrow
\mathfrak{u}_{P'}=(u_{1},\ldots,u_{i-1},u_{i}-1,u_{i+1},\ldots,u_{an}),
\end{align}
for some $i\in[2,an]$ such that $u_{i-1}<u_{i}$.
The {\it Young order} on $\mathfrak{D}_{n}^{(a,b)}$ is defined 
from the cover relation $\lessdot_{\mathrm Y}$ by the transitive and 
reflexive closure of it.

\begin{example}
We consider the same rational Dyck path $P$ in Figure \ref{fig:rDyck}.
When $P\lessdot_{\mathrm{Y}}P'$, the step sequence for $P'$ 
is either $(0,0,2,4)$, $(0,1,1,4)$, or $(0,1,2,3)$.
\end{example}

\begin{remark}
The Young order can be rephrased as follows. 
Given a rational Dyck path $P$, the shape of a Young diagram above $P$
is nothing but the step sequence of $P$.
Thus, a path $P'$ such that $P\lessdot_{\mathrm Y}P'$ is obtained by 
deleting a single box from the Young diagram.

In general, the Young order is finer than the rotation order. 
When we have $P\lessdot_{\mathrm{rot}}P'$, we have a sequence of 
Dyck paths such that 
$P\lessdot_{Y}P_1\lessdot_{Y}P_2\lessdot_{Y}\ldots P_{k}\lessdot_{Y}P'$
for some $k$.
\end{remark}

\subsection{The horizontal strip-decomposition of rational Dyck paths}
We generalize the strip decomposition \cite{KalMuh15} for $(1,b)$-Dyck paths to the 
decomposition for $(a,b)$-Dyck paths.

Let $P\in\mathfrak{D}_{n}^{(a,b)}$ be a rational Dyck path 
and $\mathfrak{h}_{P}=(h_1,h_2,\ldots,h_{bn})$ be its height sequence. 
We define a sequence $\mathfrak{H}:=(H_{1},\ldots,H_{abn})$ of positive integers of length 
$abn$ constructed from $\mathfrak{h}_{P}$ by
\begin{align*}
H_{(i-1)a+j}:=h_{i}, \quad \forall i\in[1,bn], \quad \forall j\in[1,a].
\end{align*}
For $i\in[1,b]$, we consider the sequence 
\begin{align*}
\mathfrak{H}_{i}:=(H_{i},H_{i+b},\ldots,H_{i+(n-1)b}).
\end{align*}
Then, it is easy to show that each $\mathfrak{H}_{i}$ for $1\le i\le b$ 
is the height sequence of some Dyck path $q_{i}\in\mathfrak{D}_{an}^{(1,1)}$.
We associate $b$ Dyck paths in $\mathfrak{D}_{an}^{(1,1)}$ with the height 
sequence of $P$ by the following definition.
\begin{defn}
\label{defn:stripdec}
Let $P\in\mathfrak{D}_{n}^{(a,b)}$ be a rational Dyck path associated with the 
height sequence $\mathfrak{h}_{P}$.
We consider the sequence $\delta(P):=(q_1,\ldots,q_b)$ where the path 
$q_{i}\in\mathfrak{D}_{an}^{(1,1)}$ is characterized by the height sequence
$\mathfrak{H}_{i}$ given as above.
We call the sequence $\delta(P)$ the {\it horizontal strip-decomposition} of $P$.
We denote by $\delta$ the map from $\mathfrak{D}_{n}^{(a,b)}$ to 
$(\mathfrak{D}_{an}^{(1,1)})^{b}$.
\end{defn}

\begin{remark}
The map $\delta: \mathfrak{D}_{n}^{(a,b)}\rightarrow (\mathfrak{D}_{an}^{(1,1)})^{b}$ 
defined in Definition \ref{defn:stripdec} is injective
since different Dyck paths have different height sequences 
in $\mathfrak{D}_{n}^{(a,b)}$.
\end{remark}

\begin{prop}
Let $\delta(P):=(q_1,\ldots,q_b)$ be $b$ Dyck paths obtained from 
$P\in\mathfrak{D}_{n}^{(a,b)}$.
Then, we have $q_i<q_{i+1}$ in the Young order for all $1\in[1,b-1]$.
\end{prop}
\begin{proof}
Let $h_1:=(h_{1,1},\ldots,h_{1,an})$ and $h_2:=(h_{2,1},\ldots,h_{2,an})$ 
be the height sequences of $q_{i}$ and $q_{i+1}$ respectively.
By construction of $\mathfrak{H}_{1}$ and $\mathfrak{H}_{2}$,
we have $h_{1,j}\le h_{2,j}$ for all $j\in[1,an]$.
Then, it is obvious that $q_{i}<q_{i+1}$ in the Young order.
\end{proof}

It is clear from the definition that when $a=1$, the horizontal strip-decomposition 
is exactly the same as the strip-decomposition studied in \cite{KalMuh15}.

\begin{example}
We consider the same example as Figure \ref{fig:rDyck}.
Since $P=NENENEENEE$, $(a,b)=(2,3)$ and $\mathfrak{h}_{P}=(1,2,3,3,4,4)$,
we have three Dyck paths in $(\mathfrak{D}_{4}^{(1,1)})^{3}$ whose height 
sequences are 
\begin{align*}
(1,2,3,4),\quad (1,3,3,4),\quad (2,3,4,4).
\end{align*}
\end{example}

\subsection{The vertical strip-decomposition of rational Dyck paths}
Let $P\in\mathfrak{S}_{n}^{(a,b)}$ be a rational Dyck path with 
step sequence $\mathfrak{u}_{P}:=(u_{1},\ldots,u_{an})$.

We define a sequence $(U_1,\ldots,U_{abn})$ of non-negative integers 
of length $abn$ by 
\begin{align}
U_{(i-1)b+j}:=u_{i}, \quad \forall i\in[1,an],\quad \forall j\in[1,b].
\end{align}
For $i\in[1,an]$, we consider $a$ step sequences 
$\overline{\mathfrak{u}}_{i}(P):=(\overline{u}_{i,1},\ldots,\overline{u}_{i,n})$ 
with $1\le i\le a$ by 
\begin{align*}
\overline{u}_{i,j}:=U_{a(j-1)+i},
\end{align*}
for $1\le j\le bn$.

\begin{defn}
\label{defn:vsd}
Let $\mathfrak{u}_{P}$ and $\overline{\mathfrak{u}}_{i}(P):=(\overline{u}_{i,1},\ldots,\overline{u}_{i,bn})$ 
with $1\le i\le a$ as above.
We denote by $\theta:\mathfrak{D}_{n}^{(a,b)}\rightarrow(\mathfrak{D}_{bn}^{(1,1)})^{a}$ 
the map from $\mathfrak{u}_{P}\in\mathfrak{D}_{n}^{(a,b)}$ to 
$\{\overline{\mathfrak{u}}_{i}(P)\}_{i=1}^{a}\in(\mathfrak{D}_{bn}^{(1,1)})^{a}$.
We call the sequence $\theta(P):=\theta(\mathfrak{u}_{P})$ the vertical 
strip-decomposition of $P$.
\end{defn}

\begin{remark}
The map $\theta:\mathfrak{D}_{n}^{(a,b)}\rightarrow(\mathfrak{D}_{bn}^{(1,1)})^{a}$ defined 
in Definition \ref{defn:vsd} is injective as in the case of the horizontal decomposition.
\end{remark}

\begin{prop}
Let $\theta(P):=(q_1,\ldots,q_q)$ be $b$ Dyck paths obtained from 
$P\in\mathfrak{D}_{n}^{(a,b)}$ by the vertical strip-decomposition.
Then, we have $q_i\ge q_{i+1}$ in the Young order for all $1\in[1,b-1]$.
\end{prop}
\begin{proof}
Let $u_1:=(u_{1,1},\ldots,u_{1,bn})$ and $u_2:=(u_{2,1},\ldots,u_{2,bn})$ 
be the height sequences of $q_{i}$ and $q_{i+1}$ respectively.
By construction of $u_{1}$ and $u_{2}$,
we have $u_{1,j}\le u_{2,j}$ for all $j\in[1,bn]$.
Then, it is obvious that $q_{i}\ge q_{i+1}$ in the Young order.
\end{proof}

\begin{example}
We consider the same rational Dyck path $P$ as Figure \ref{fig:rDyck}.
Since $(a,b)=(2,3)$ and step sequence $\mathfrak{u}_{P}=(0,1,2,4)$, 
we have two rational Dyck paths in $(\mathfrak{D}_{6}^{(1,1)})^{2}$ whose 
step sequences are 
\begin{align*}
(0,0,1,2,2,4), \quad (0,1,1,2,4,4).
\end{align*}
\end{example}

\section{Multi-permutations}
\label{sec:multip}
In this section, we study the $b$-Stirling permutation and its orders.

\subsection{\texorpdfstring{$b$-Stirling permutations}{b-Stirling permutations}}
A {\it $b$-Stirling permutation of size $n$} is a permutation of the multiset
$\{1^b,2^b,\ldots,n^b\}$ such that 
if an integer $j$ appears between two $i$'s, then we have $j>i$ \cite{GesSta78,Par94a,Par94b}.
We denote by $\mathfrak{S}_{n}^{(b)}$ the set of $b$-Stirling 
permutations of size $n$.
As in the case of permutations, we adapt the one-line notation 
such that $\pi:=(\pi_1,\ldots,\pi_{bn})\in\mathfrak{S}_{n}^{(b)}$.

\begin{remark}
When $b=2$, the number of $2$-Stirling permutation of size $n$ 
is given by the double factorial $(2n-1)!!$.
\end{remark}

We define a map $\zeta:\mathfrak{D}_{n}^{(a,b)}\rightarrow\mathfrak{S}_{an}^{(b)}$.
\begin{defn}
\label{defn:zeta}
Given a path $P$ associated with the step sequence $\mathfrak{u}_{P}$, 
we define $\pi\in\mathfrak{S}_{n}^{(b)}$ from $\mathfrak{u}_{P}$ recursively as follows.
\begin{enumerate}
\item Set $i=1$ and $\pi:=\emptyset$.
\item We insert an integer sequence $i^{b}$ into the $u_{i}$-th position (from left)
of $\pi$.  
\item We increase $i$ by one, and go to (2).
The algorithm stops when we obtain $\pi\in\mathfrak{S}_{n}^{(b)}$ for $i=n$.
\end{enumerate}
We denote by $\zeta$ the map from $\mathfrak{D}_{n}^{(a,b)}$ to $\mathfrak{S}_{n}^{(b)}$, 
and denote $\pi$ by $\zeta(\mathfrak{u}_{P})$.
\end{defn}

\subsection{Rooted trees for Dyck paths}
\label{sec:rtree}
Let $q\in\mathfrak{D}_{n}^{(1,1)}$ be a Dyck path of size $n$.
A Dyck path consists of $n$ up steps and $n$ right steps.
We construct a rooted tree from $q$ following \cite{LS81}.
Let $N$ (resp. $E$) be an up (resp. down) step in $q$, 
and $Z$ be a set of Dyck paths of arbitrary size.
\begin{enumerate}
\item A single edge consisting of the root and one leaf corresponds to a Dyck path $NE$. 
\item Let $q'$ and $q"$ be Dyck paths in $Z$. 
When $q$ is a concatenation of $q'$ and $q"$, the tree for $q$ is obtained by attaching 
the trees for $q'$ and $q''$ at their roots.
\item Let $q'$ be a Dyck path and $q=Nq'E$. The tree for $q$ is obtained by attaching 
an edge just above the tree for $q'$. 
\end{enumerate}
We denote the tree for a Dyck path $q$ by $T(q)$.

Let $T(q)$ be a tree for a Dyck path $q$ with $n$ edge.
\begin{defn}
We put a positive integer in $[1,n]$ on edges in $T$ such that 
\begin{enumerate}
\item integers are increasing from the root to a leaf,
\item each integer in $[1,n]$ appears exactly once.
\end{enumerate}
We call this labeling a natural label $L$.
\end{defn}

From a natural label $L$, we construct two words $w_{\mathrm{pre}}(L)$ 
and $w_{\mathrm{post}}(L)$ according to the pre-order and the post-order 
respectively.
\begin{defn}
The pre-order word $w_{\mathrm{pre}}(L)$ is a word obtained from $L$ 
by reading the labels of edges in the following three steps:
(1) visit the root edge, 
(2) traverse the left subtree, and (3) traverse the right subtree.

Similarly, the post-order word $w_{\mathrm{post}}(L)$ is a word obtained from $L$ 
by reading the labels of edges in the following three steps:
(1) traverse the left subtree, (2) traverse the right subtree, and 
(3) visit the root edge.
\end{defn}

\subsection{\texorpdfstring{$(b+1)$-ary tree}{(b+1)-ary trees}}
\label{sec:mptree}
In this section, we summarize $(b+1)$-ary trees and $b$-Stirling 
permutations. We give a correspondence between $(b+1)$-ary trees 
and $b$-Stirling permutations (see for example \cite{CebGonDLe19} 
and references therein).

A tree is a graph consisting of nodes and edges without loops and cycles.
A tree is said to be rooted if one of its nodes is specially called the {\it root}.
A node $x$ is said to be the {\it parent} (reps. {\it child}) of the node $y$ if $x$ is the node 
just below (resp. above) the node $y$.
A node is called a {\it leaf} if it has no children and called {\it internal} if it has children.

A $(b+1)$-ary tree is a rooted tree such that each internal node has exactly $b+1$ children.
We denote by $\mathfrak{T}_{n}^{(b)}$ the set of $(b+1)$-ary trees with $n$ nodes.
We say that an element in $\mathfrak{T}_{n}^{(b)}$ is size $n$.
Below, by defining the map $\xi$ from $\mathfrak{S}_{n}^{(b)}$ to $\mathfrak{T}_{n}^{(b)}$, 
we show that $|\mathfrak{T}_{n}^{(b)}|=|\mathfrak{S}_{n}^{(b)}|$, {\it i.e.,} $\xi$ 
is bijective.

Let $\pi\in\mathfrak{S}_{1}^{(b)}$. Then we have a unique $b$-Stirling permutation $\pi=1^{b}$.
The tree $\xi(\pi)$ is defined as a $(b+1)$-ary tree with a single internal node.
We label the regions between two edges in $\xi(\pi)$ by $1$. Note that since we have $b+1$ edges,
the number of label $1$ is $b$.

We recursively construct a $(b+1)$-ary tree for $\pi\in\mathfrak{S}_{n}^{(b)}$ starting 
from a permutation $1^{b}$.
Since there is no $1\le i<n$ between two $n$'s, $b$ $n$'s appear as a single 
block in $\pi$.
We remove all $n$'s from $\pi$ and denote by $\pi'$ the new $b$-Stirling permutation of 
length $n-1$.
We denote by $\xi(\pi')$ the $(b+1)$-ary tree of size $n-1$ for $\pi'$.
Let $x$ be the position of the left-most $n$ in $\pi$ from left.
Since we have $(b+1)(n-1)$ leaves in the $(b+1)$-ary tree $\xi(\pi')$, 
we append a $(b+1)$-ary tree $T_{1}$ of size one at the $x+1$-th leaf from left. 
We label the regions between two edges in $T_{1}$ by $n$.
We denote by $\xi(\pi)$ the new tree of size $n$.

It is obvious that the above procedure is invertible. 
Therefore, we have 
\begin{prop}
The map $\xi:\mathfrak{S}_{n}^{(b)}\rightarrow\mathfrak{T}_{n}^{(b)}$ is 
bijective.
\end{prop}

\subsection{Rotation order}
\label{sec:mprot}
We define an operation $\mathfrak{r}_{i}:\mathfrak{T}_{n}^{(b)}\rightarrow\mathfrak{T}_{n}^{(b)}$
for $i\in[2,n]$ as follows.
Let $T$ be a $(b+1)$-ary tree of size $n$.
Recall that each region between edges in $T$ is labeled by $j\in[1,n]$.
Find the subtree in $T$ such that the regions between edges connected to the root 
are labeled by $i$.
By construction, such $i$ always exists.
We denote by $T_{i}$ the subtree obtained as above and $T\setminus T_{i}$ be 
a tree obtained by deleting $T_{i}$ from $T$.
If there exists a subtree $T_{j}$ with smallest labels $j$ 
such that the root of $T_j$ is connected 
to the right-most edge with label $i$, we denote by $T_{i}\setminus T_{j}$
the subtree obtained from $T_i$ by deleting $T_{j}$.

Suppose that the subtree $T_{i}$ is appended $x$-th leaf from left in 
$T\setminus T_{i}$.
Let $h$ be the label which characterize the edge connected to the $x-1$-th leaf in $T\setminus T_{i}$.
We append the tree $T_{i}\setminus T_{j}$ on the $x-1$-th leaf from left in $T\setminus T_{i}$ if $h<i$.
Then, we append the tree $T_j$ on the $x$-th leaf from left in $T\setminus T_{i}$.
We denote by $T_{\mathrm{new}}$ the new tree.

We define the map $\mathfrak{r}_{i}$ as $T\mapsto T_{\mathrm{new}}$.
If $h>i$, we define $\mathfrak{r}_{i}$ as an identity, that is, $\mathfrak{r}_{i}(T)=T$.
We call $\mathfrak{r}_{i}$ with $i\in[2,n]$ a {\it rotation} of $T$.

\begin{example}
The action of $\mathfrak{r}_{2}$ on a tree is given in Figure \ref{fig:treerot}.
A circled number $i=3$ or $4$ corresponds to a tree with three edges with label $i$.
\begin{figure}[ht]
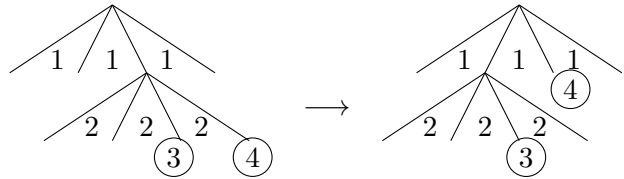

\tikzpic{-0.5}{[scale=0.6]
\coordinate
	child{coordinate(c1)}
	child{coordinate(c2)}
	child{coordinate(c3)
		child{coordinate(c5)}
		child{coordinate(c6)}
		child{coordinate(c7)}
		child{coordinate(c8)}
	}
	child{coordinate(c4)};
\draw($(c3)!0.8!(c7)$)node[anchor=north]{\cnum{3}};
\draw($(c3)!0.8!(c8)$)node[anchor=north west]{\hspace*{-5mm}\cnum{4}};
\draw($(0,0)!0.8!($(c1)!0.5!(c2)$)$)node{$1$};
\draw($(0,0)!0.8!($(c3)!0.5!(c2)$)$)node{$1$};
\draw($(0,0)!0.8!($(c3)!0.5!(c4)$)$)node{$1$};
\draw($(c3)!0.8!($(c5)!0.5!(c6)$)$)node{$2$};
\draw($(c3)!0.8!($(c6)!0.5!(c7)$)$)node{$2$};
\draw($(c3)!0.8!($(c7)!0.5!(c8)$)$)node{$2$};
}$\longrightarrow$
\tikzpic{-0.5}{[scale=0.6]
\coordinate
	child{coordinate(c1)}
	child{coordinate(c2)
		child{coordinate(c5)}
		child{coordinate(c6)}
		child{coordinate(c7)}
		child{coordinate(c8)}
	}
	child{coordinate(c3)}
	child{coordinate(c4)};
\draw($(c3)!0.8!(c7)$)node[anchor=north]{\cnum{3}};
\draw($(0,0)!0.8!(c3)$)node[anchor=north west]{\hspace*{-5mm}\cnum{4}};
\draw($(0,0)!0.8!($(c1)!0.5!(c2)$)$)node{$1$};
\draw($(0,0)!0.8!($(c3)!0.5!(c2)$)$)node{$1$};
\draw($(0,0)!0.8!($(c3)!0.5!(c4)$)$)node{$1$};
\draw($(c2)!0.8!($(c5)!0.5!(c6)$)$)node{$2$};
\draw($(c2)!0.8!($(c6)!0.5!(c7)$)$)node{$2$};
\draw($(c2)!0.8!($(c7)!0.5!(c8)$)$)node{$2$};
}
\caption{Action of $\mathfrak{r}_{2}$ on a tree.}
\label{fig:treerot}
\end{figure}
Since the subtree with label $4$ is connected to the right-most edge with label $2$,
this subtree is reconnected to an edge with label $1$ in the right picture.
\end{example}

Let $\pi,\pi'\in\mathfrak{S}_{n}^{(b)}$ be two $b$-Stirling permutations, and 
$T_{\pi}$ and $T_{\pi'}$ be two $(b+1)$-ary trees corresponding to 
$\pi$ and $\pi'$ respectively.

We define a cover relation $\lessdot_{\mathrm{rot}}$ on $\pi$ and $\pi'$
by
\begin{align}
\pi\lessdot_{\mathrm{rot}}\pi'\Leftrightarrow
T_{\pi'}=\mathfrak{r}_{i}(T_{\pi}),
\end{align}
for some $i\in[2,n]$ such that $T_{\pi'}\neq T_{\pi}$.
We also denote the relation on trees by $T_{\pi}\lessdot_{\mathrm{rot}}T_{\pi'}$. 
The rotation order on $\mathfrak{S}_{n}^{(b)}$  is defined from the 
cover relation $\lessdot_{\mathrm{rot}}$ by the transitive and 
reflexive closure of it.
Since we consider $b$-Stirling permutations and $(b+1)$-ary trees, 
the rotation $\lessdot_{\mathrm{rot}}$ on $(b+1)$-ary trees 
corresponds to the rotation of rational Dyck paths with $(a,b)=(1,b)$.
\begin{prop}
\label{prop:rottree1}
The rotation order on the tree coincides with the rotation order 
on the $b$-Stirling permutations. 
Let $P$ and $P'$ be rational Dyck paths in $\mathfrak{D}_{n}^{(1,b)}$ 
and $T_{\pi}$ and $T_{\pi'}$ be $(b+1)$-ary trees corresponding 
to $P$ and $P'$ respectively, {\it i.e.}, $T_{\pi}=\xi\circ\zeta(P)$ 
and $T_{\pi'}=\xi\circ\zeta(P')$.
Then, we have
\begin{align}
\label{eqn:rotbt2}
P\lessdot_{\mathrm{rot}}P'\Leftrightarrow
T_{\pi}\lessdot_{\mathrm{rot}}T_{\pi'}.
\end{align}
\end{prop}
\begin{proof}
From Proposition \ref{prop:rot2} and the condition $(a,b)=(1,b)$, 
it is enough to prove that 
$P\lessdot_{P_0}P'\Leftrightarrow T_{\pi}\lessdot_{\mathrm{rot}}T_{\pi'}$.
To pick an integer $i$, which characterizes a subtree $T_{i}$, corresponds to 
the choice of a lattice point $p_1$ which is preceded by 
an $E$ step and followed by a $N$ step in $P$.
To perform a rotation in $P$, we need to know the lattice point $p_2$ 
such that the horizontal distance $\mathrm{hor}_{P_0}(p_2)$ is 
equal to $\mathrm{hor}_{P_0}(p_1)$ and it is the first lattice point after
$p_1$ in $P$.
Recall that the subtree $T_{j}$ is connected to the right-most edge with a 
label $i$ in $T_{i}$.
Since we consider a rational Dyck path with $(a,b)=(1,b)$, the choice 
of $j$ corresponds to the choice of the lattice point $p_2$ in $P$.
Actually, we have $(j-i)b$ edges in $T_{i}\setminus T_{j}$ with labels 
in $[i,j-1]$, which implies the horizontal distance of $p_2$ is the same 
as that of $p_1$.
Further, the switch of the $E$ step and a subpath $P[p_1,p_2]$ in $P$ 
is equivalent to move the subtree $T_{i}\setminus T_{j}$ loft by one 
step and connect $T_{j}$ on $T\setminus T_{i}$ at $x$-th leaf of it.
From these observations, this completes the proof of Eqn. (\ref{eqn:rotbt2}).
\end{proof}

\subsection{Young order}
Let $\pi,\pi'\in\mathfrak{S}_{n}^{(b)}$ be two $b$-Stirling permutations.

Suppose that $\pi:=(\pi_{1},\ldots,\pi_{bn})$ satisfies 
$\pi_{i}:=r<\pi_{i+1}$. 
Let $s$ be a minimal integer such that 
$r<s$ and all $b$ $r$'s appear as $\pi_{j}$ with $j\ge i+1$.
The $2b$ integers $r$'s and $s$'s form a subsequence in $\pi$
$r^{k}s^{b}r^{b-k}$ with some $k\in[1,b]$.
Note that $k\ge 1$ since $\pi_{i}=r$.
We define the map $\chi_{r,s}$ acting on the sequence $r^{k}s^{b}r^{b-k}$ as 
\begin{align*}
\chi_{r,s}: r^{k}s^{b}r^{b-k}\mapsto r^{k-1}s^{b}r^{b-k+1}.
\end{align*}
We naturally extend the action of $\chi_{r,s}$ to a $b$-Stirling permutation $\pi$, 
that is, we change the partial sequence consisting of $r$'s and $s$'s by $\chi_{r,s}$ and 
do not change the remaining sequence.
Then, we can define a map 
$\chi_{r,s}:\mathfrak{S}_{n}^{(b)}\rightarrow\mathfrak{S}_{n}^{(b)}, \pi\mapsto\pi'$.

We define a cover relation $\lessdot_{Y}$ on $\pi$ and $\pi'$ by 
\begin{align}
\pi\lessdot_{Y}\pi'\Leftrightarrow
\pi'=\chi_{r,s}(\pi),
\end{align}
for some $i\in[1,bn-1]$ and $(\pi_{i},\pi_{i+1})=(r,s)$.
The {\it Young order} on $\mathfrak{S}_{n}^{(b)}$ is defined from the 
cover relation $\lessdot_{Y}$ by the transitive and reflexive closure of it.

\begin{prop}
The Young order on the $b$-Stirling permutations coincides with the Young order 
on the rational Dyck paths. 
Let $P$ and $P'$ be rational Dyck paths in $\mathfrak{D}_{n}^{(1,b)}$
and $\pi$ and $\pi'$ be $b$-Stirling permutations corresponding 
to $P$ and $P'$ respectively.
Then, we have
\begin{align}
\label{eqn:ytree}
P\lessdot_{\mathrm{Y}}P'\Leftrightarrow
\pi\lessdot_{\mathrm{Y}}\pi'.
\end{align}
\end{prop}
\begin{proof}
Given a Dyck path $P$, we construct $\pi$ by use of the step sequence 
$\mathfrak{u}_{P}:=(u_1,\ldots,u_{n})$ of $P$.
Each entry $u_{i}$ of the step sequence encodes the positions of the left-most
integer with label $i$.
The action of $\chi_{r,s}$ moves the positions of integers with label $s$
left by one step and does not change the positions of integers with other 
labels.
This is nothing but to decrease $u_{s}$ by one in the step sequence.
From these observations, we have Eqn. (\ref{eqn:ytree}).
\end{proof}

\section{\texorpdfstring{$312$-avoiding Stirling permutations}{312-avoiding Stirling permutations}}
\label{sec:312}
\subsection{\texorpdfstring{$312$-avoiding Stirling permutations}{312-avoiding Stirling permutations}}
Let $m$ be a positive integer and fix it through this subsection.
In this subsection, we consider the $m$-Stirling permutations $\mathfrak{S}_{n}^{(m)}$ 
for a rational Dyck path $\mathfrak{D}_{n}^{(1,m)}$.

Let $\pi\in\mathfrak{S}_{n}^{(m)}$ be a multi-permutation.
We say that $\pi$ avoids a pattern $312$ if we have no three integers $i<j<k$ 
such that $\pi_{j}<\pi_{k}<\pi_{i}$.
We denote by $\mathfrak{S}_{n}^{(m)}(312)$ the set of multi-permutations 
avoiding the pattern $312$ in $\mathfrak{S}_{n}^{(m)}$.

We introduce an order on $\mathfrak{S}_{n}^{(m)}$, which we call a {\it rotation order}.
This rotation order restricted to $\mathfrak{S}_{n}^{(m)}(312)$ 
is shown to be the same as the one in Section \ref{sec:mprot}. 
Let $\pi,\pi'\in\mathfrak{S}_{n}^{(m)}$ be multi-permutations 
given by $\pi=(\pi_1,\ldots,\pi_{n})$ and $\pi'=(\pi'_1,\ldots,\pi'_{n})$.
The order is defined as follows.
\begin{defn}
\label{defn:rot2}
If $\pi'$ is obtained from $\pi$ by applying a rotation to numbers 
in positions $i$ and $j$, that is, 
\begin{align}
\label{eqn:rotSperm}
\pi'=(\pi_1,\ldots, \pi_{i-1},\pi_{i+1},\ldots,\pi_{j},\pi_{i},\ldots,\pi_{n}),
\end{align}
where $i$ and $j$ satisfy the following conditions:
\begin{enumerate}
\item $\pi_{i}<\pi_{i+1}$, 
\item $\pi_{k}\ge \pi_{i+1}$ for all $k\in[i+1,j]$,  
\item $\pi_{k}\neq \pi_{i+1}$ for all $k\in[j+1,n]$,
\item $\pi_{j}=\pi_{i+1}$.
\end{enumerate}
When $\pi$ and $\pi'$ satisfy the above conditions, we denote the cover relation 
by $\pi\lessdot_{R}\pi'$. 

The rotation order $<_{R}$ is defined as the transitive and reflexive 
closure of the cover relation $\lessdot_{R}$.
\end{defn}

Let $P\in\mathfrak{S}_{n}^{(1,b)}$ be a rational Dyck path associated with 
step sequence $\mathfrak{u}_{P}:=(u_1,\ldots,u_{n})$.

We construct a multi-permutation $\pi(P)$ in $\mathfrak{S}_{n}^{(m)}$ from 
$\mathfrak{u}_{P}$ by $\zeta(\mathfrak{u}_{P})$ 
with a substitution $b=m$ in the process (1), (2) and (3) of Definition \ref{defn:zeta}.

\begin{lemma}
\label{lemma:Dyck312}
A multi-permutation $\pi(P)$ for a $m$-Dyck path $P$ is 
in $\mathfrak{S}_{n}^{(m)}(312)$.
\end{lemma}
\begin{proof}
The permutation $\pi(P)$ is constructed from the step sequence 
$\mathfrak{u}_{P}:=(u_1,\ldots,u_{n})$ for the $m$-Dyck path $P$.

We prove the statement by induction on $n$. For $n=1$, we have a sequence 
$1^{m}$ and it avoids the pattern $312$. We assume that lemma is true 
for $n=i-1$.
Recall that the sequence $\mathfrak{u}_{P}$ is a weakly increasing sequence by 
construction.
By step (2) in Definition \ref{defn:zeta}, we insert $i^{m}$ into 
the $u_{i}$-the position from left.
Since $u_{i-1}\le u_{i}$, the subsequence $i^{m}$ is at most left to the 
left-most $i-1$, or right to the left-most $i-1$.
Therefore, we have no $312$ pattern after the insertion of $i^{m}$.
By induction, we have no $312$ pattern in $\pi(P)$, which completes the proof.
\end{proof}

\begin{prop}
The rotation order in Definition \ref{defn:rot2} 
for $\mathfrak{S}_{n}^{(m)}(312)$ is equivalent 
to the rotation order in Section \ref{sec:rot}.
\end{prop}
\begin{proof}
From Proposition \ref{prop:rottree1}, it is enough to
show that the rotation order in Definition \ref{defn:rot2}
is equivalent to the rotation on $(m+1)$-ary trees.
The conditions (1) and (4) correspond to fixing a subtree 
$T_{i}$ for some $i$, and the rotation given 
by Eqn. (\ref{eqn:rotSperm}) corresponds to moving $T_i$ 
left by one edge in the tree.
Recall that the subtree $T_j$ is reconnected to the subtree 
$T\setminus T_{i}$.
This operation is naturally realized by the 
rotation (\ref{eqn:rotSperm}), since $\pi_k$ with $k\in[j+1,n]$ 
remains as the same position.
Therefore, these two rotations give the same cover relation.
\end{proof}

\subsection{Presentation by multiple Dyck words }
\label{sec:pmDw}
Given a rational Dyck path $P\in\mathfrak{D}_{n}^{(1,m)}$ with $m\ge1$, we associate it with 
$m$ Dyck words.
We give a correspondence between $m$-Dyck paths and parenthesis presentation 
following \cite{CebGonDLe19}.
By making use of this parenthesis presentation, we give an $m$-tuple of Dyck words, which 
we call the parenthesis presentation of type $I$. 
Type $II$ will be introduced in Section \ref{sec:mockDw}.

Let $\pi:=\pi(P)\in\mathfrak{S}_{n}^{(m)}$ be a multi-permutation and 
$T(\pi)$ be its corresponding $(m+1)$-ary tree constructed from 
the step sequence $\mathfrak{u}_{P}$ by the map $\zeta$ (see Definition \ref{defn:zeta}).
We will construct a parenthesis presentation with $\ast$ as follows.
We walk around $T(P)$ from the root in anti-clockwise.
When we walk along the left-most (resp. right-most) edge labeled by $i$ for $1\le i\le n$,
we write ``$($" (resp. ``$)$").
Similarly, when we walk the region between edges labeled by $i$, 
we write $\ast$.
Thus, we have a word consisting of ``$($", ``$)$" and $\ast$, and denote it 
by $\alpha^{\ast}(\pi)$.
For $\pi(P)\in\mathfrak{S}_{n}^{(m)}$, we construct an $m$-tuple of parenthesis presentations 
of $P$ from $\alpha^{\ast}(\pi)$.
We denote by $\mathfrak{A}^{\ast}_{n,m}$ the set of parenthesis presentation with $\ast$ consisting 
of $n$ ``$($"'s, $n$ ``$)$"'s and $nm$ $\ast$'s.

We denote by $\mathfrak{A'}_{n}^{m}$ the set of an $m$-tuple of Dyck paths $\mathfrak{a}:=(a_{1},\ldots,a_{m})$
of length $n$ such that $\alpha_m\le_{Y}\alpha_{m-1}\le_{Y}\ldots\le_{Y}\alpha_1$ in the Young order.

There are $m$ $\ast$'s between the left parenthesis ``$($" and 
the right parenthesis ``$)$" labeled by the same integer.
We enumerate $m$ $\ast$'s between two parentheses ``$($" and ``$)$" from right to left  by $1,\ldots,m$.

\begin{defn}
\label{defn:alphaast}
Let $\alpha^{\ast}(\pi)$ be a parenthesis presentation of $\pi$.
We define a Dyck word $\alpha_{i}^{I}(\pi)$ for $1\le i\le m$ consisting of $($'s and $)$'s  
by the following processes:
\begin{enumerate}
\item Erase all ``$)$"'s and all $\ast$'s but $\ast$ labeled by $i$ 
in $\alpha^{\ast}(\pi)$. 
\item Replace the remaining $\ast$'s by ``$)$"'s.
\end{enumerate}
We denote by $\alpha^{I}$ the map from $\mathfrak{S}_{n}^{(m)}$ to $\mathfrak{A'}_{n}^{m}$, 
and call it the parenthesis presentation of type $I$.
\end{defn} 

\begin{example}
\label{ex:alpha}
Let $\pi=122133$ for $m=2$. 
We have $\alpha^{\ast}(\pi)=(\ast(\ast\ast)\ast(\ast\ast))$ and 
\begin{align*}
&\alpha_{1}^{I}(\pi)=(())(), \\
&\alpha_{2}^{I}(\pi)=()()().
\end{align*}
\end{example}

\begin{prop}
\label{prop:312avoid}
Let $\alpha^{I}:\mathfrak{S}_{n}^{(m)}\rightarrow\mathfrak{A'}^{m}_{n}$ be a 
parenthesis presentation of type $I$.
Let $\pi,\pi'\in\mathfrak{S}_{n}^{(m)}(312)$ be two $m$-Stirling permutations.
Then we have $\alpha^{I}(\pi)\neq\alpha^{I}(\pi')$ if $\pi\neq\pi'$. 
In other words, $\alpha^{I}$ is a injective on  
$\mathfrak{S}_{n}^{(m)}(312)\subset\mathfrak{S}_{n}^{(m)}$.
\end{prop}
\begin{proof}
We prove the statement by induction on $n$.
When $n=2$, we have $m+1$ $m$-Stirling permutations 
and $|\mathfrak{A'}^{m}_{n}|=m+1$.
We assume that the statement is true up to $n-1$.

Let $\pi_i$ be a $m$-Stirling permutation obtained from $\pi$ by 
deleting integers in $[i+1,n]$.
Since $\pi\neq\pi'$, there exists an integer $i$ such that 
$\pi_{j}=\pi'_{j}$ for $1\le j\le i-1$ and $\pi_i\neq\pi'_{i}$.
If $i\le n-1$, $\alpha^{I}(\pi_{i})\neq\alpha^{I}(\pi'_{i})$ by induction 
assumption. This implies that $\alpha^{I}(\pi)\neq\alpha^{I}(\pi')$. 
When $i=n$, we consider a permutations $\widetilde{\pi}$ 
(resp. $\widetilde{\pi'}$) obtained from $\pi$ (resp. $\pi'$) by deleting 
$m$ integer $1$'s.
This permutation is size $n-1$ and by induction assumption, we 
have $\alpha^{I}(\widetilde{\pi})\neq\alpha^{I}(\widetilde{\pi'})$.
This implies  $\alpha^{I}(\pi)\neq\alpha^{I}(\pi')$, which completes the proof.
\end{proof}

\begin{remark}
Let $m=2$, $\pi=122133$ and $\pi'=133122$ be two $2$-Stirling permutations.
Then, we have $\alpha^{I}(\pi)=\alpha^{I}(\pi')=(\ast(\ast\ast)\ast(\ast\ast))$. 
The permutation $\pi'$ contains the pattern $312$.
This is one of the reasons why we need a condition of the pattern avoidance in 
Proposition \ref{prop:312avoid}. 
\end{remark}

If we replace $($'s and $)$'s by $N$ and $E$ in $\alpha_i^{I}(\pi)$ for 
$1\le i\le m$, we have a Dyck path in $\mathfrak{S}_{n}^{(1,1)}$.
By abuse of notation, we also denote by $\alpha^{I}_i$ the corresponding 
Dyck path. 

\begin{prop}
Let $\alpha_{i}^{I}(\pi)$ for $1\le i\le m$ be Dyck words defined as above.
Then, $\alpha^{I}_{i+1}\le_{Y}\alpha^{I}_{i}$ as Dyck paths. 
\end{prop}
\begin{proof}
The Dyck word $\alpha^{I}_{i}$ is obtained from $\alpha^{\ast}(\pi)$ by replacing 
the $i$-th $\ast$ from right with a right parenthesis ``$)$". 
By construction, the $i+1$-th $\ast$ is left to the $i$-th $\ast$ in $\alpha^{\ast}$.
This implies that $\alpha^{I}_{i+1}<_{Y}\alpha^{I}_{i}$ for all $i\in[1,m]$.
\end{proof}

We will construct a multi-permutation $\pi\in\mathfrak{S}_{n}^{(m)}$ from 
$\alpha^{I}_i$ for $1\le i\le m$.
By regarding $\alpha^{I}_{i}$ as a Dyck path, we have a tree $T_{i}:=T(\alpha^{I}_{i})$ 
as in Section \ref{sec:rtree}.
Let $L^{\mathrm{pre}}_{i}$ be a natural label of $T_{i}$ such that 
the pre-order word of the label is identity, namely, we have 
$w_{\mathrm{pre}}(L^{\mathrm{pre}}_{i})=12\ldots n$.
Then, we define a word $w_{i}$ from $L^{\mathrm{pre}}_{i}$ by
$w_{i}:=w_{\mathrm{post}}(L^{\mathrm{pre}}_{i})$ for $1\le i\le m$.
Then, we define a sequence of non-negative integers $\mathfrak{u}:=(u_1,\ldots,u_{n})$ by
\begin{align}
\label{eqn:ufromd}
u_{j}:=\sum_{i=1}^{m}d_{i}(j),
\end{align}
where 
\begin{align*}
d_{i}(j):=\#\{ k<j| k \text{ is left to } j \text{ in } w_{i} \}.
\end{align*}

\begin{example}
We consider the same example as Example \ref{ex:alpha}.
The post-order words for the Dyck paths corresponding to $\alpha^{I}_1(\pi)=(())()$ and 
$\alpha^{I}_2(\pi)=()()()$ are $213$ and $123$.
From these, we have the step sequence $\mathfrak{u}=(0,1,4)$,
Applying $\zeta$ on $\mathfrak{u}$, we have $122133$.
\end{example}

\begin{lemma}
\label{lemma:mDTSS}
Let $\mathfrak{u}:=(u_1,\ldots,u_{n})$ defined from $\alpha^{I}_{i}$ for $1\le i\le m$ as above.
Then, we have $u_{1}\le u_{2}\le \ldots\le u_{n}$ and $u_{k}\le m(k-1)$ for $k\in[1,n]$.
This means that $\mathfrak{u}$ is a step sequence for some rational Dyck path $P$.
\end{lemma}
\begin{proof}
From Eqn. (\ref{eqn:ufromd}), it is clear that $u_{k}\le u_{k+1}$ since 
we have $d_{i}(k)\le d_{i}(k+1)$ for all $k$.
Further, we have $u_k\le m(k-1)$ since $d_{i}(j)\le j-1$ by definition.
These imply that $\mathfrak{u}$ is a step sequence of some Dyck path $P$.
\end{proof}

The purpose of the rest of this subsection is to give a bijection between 
$\mathfrak{S}_{n}^{(m)}(312)$ and the set $\mathfrak{A}_{n}^{m}$ of an $m$-tuple 
of Dyck paths. 
First, we study the relation between an $m$-tuple of Dyck paths denoted by $\mathfrak{a}$ 
and step sequences of rational Dyck paths in $\mathfrak{D}_{n}^{(1,m)}$.
Secondly, we introduce the notion of admissibility for $\mathfrak{a}$.
This admissibility corresponds to considering $\mathfrak{S}_{n}^{(m)}(312)$
in $\mathfrak{S}_{n}^{(m)}$.

We give a simple definition of a map from $\mathfrak{a}\in\mathfrak{A'}_{n}^{m}$ to 
$\mathfrak{u}\in\mathfrak{U}_{n}^{m}$. 
\begin{defn}
\label{defn:atou}
Let $\mathfrak{U}_{n}^{m}$ be the set of step sequence $\mathfrak{u}:=(u_1,\ldots,u_{n})$ for 
rational Dyck paths in $\mathfrak{D}_{n}^{(1,m)}$. 

Then, we define a map $\beta:\mathfrak{A'}_{n}^{m}\rightarrow \mathfrak{U}_{n}^{m}$, 
$\mathfrak{a}=(a_{1},\ldots,a_{m})\mapsto\mathfrak{u}=(u_{1},\ldots,u_{n})$ 
by 
\begin{align}
\label{eqn:atousum}
\mathfrak{u}=\sum_{j=1}^{m}\mathfrak{u}(a_{j}),
\end{align}
where $\mathfrak{u}(a_{j})$ is the step sequence associated with the Dyck path $a_{j}$.
\end{defn}

\begin{prop}
\label{prop:utoa1}
The map $\beta$ in Definition \ref{defn:atou} is well-defined. In other words,
given $\mathfrak{a}$, $\beta(\mathfrak{a})$ is a step sequence of some Dyck path 
$P\in\mathfrak{D}_{n}^{(1,m)}$.
\end{prop}
\begin{proof}
We first show that $\beta(\mathfrak{a})$ is a step sequence of some Dyck path $P$.
Since $\mathfrak{u}(a_{i})$ is a step sequence of a Dyck path in $\mathfrak{D}_{n}^{(1,1)}$,
entries in $\mathfrak{u}(a_{i})$ form a non-decreasing integer sequence.
The sum of non-decreasing integer sequences is also a non-decreasing integer sequence,
$\mathfrak{u}$ in Eqn. (\ref{eqn:atousum}) is obviously a step sequence of some Dyck 
path $P$.
\end{proof}

We give an alternative definition of $\mathfrak{a}:=(a_1,\ldots,a_{n})$ in $\mathfrak{A'}_{n}^{m}$
obtained from a step sequence $\mathfrak{u}:=(u_{1},\ldots,u_{n})$ in $\mathfrak{D}_{n}^{(1,m)}$ 
as follows. 
Suppose that $u_{i}=u_{i+1}$.
We say that the {\it second primitive subsequence of $\mathfrak{u}_{P}$ at position $i$ with 
respect to $b$}
is the unique subsequence $(u_i,u_{i+1},\ldots,u_{k})$
such that 
\begin{align}
\label{eqn:2ndps}
\begin{aligned}
&u_{j}-u_{i}\le b(j-i-1), \quad \forall j\in[i+1,k] \\
&\text{either\ } k=n,\ \text{or}\ u_{k+1}-u_{i}>b(k-i). 
\end{aligned}
\end{align}

We define a non-negative integer sequence 
$\mathfrak{l}^{b}:=(\mathfrak{l}_{1},\ldots,\mathfrak{l}_{n})$ as follows. 
\begin{enumerate}
\item In case of $u_{i}\neq u_{i+1}$. We define $\mathfrak{l}_{i}:=u_{i}$. 
\item In case of $u_{i}=u_{i+1}$. 
Let $(u_{i},\ldots,u_{k})$ be the second primitive subsequence of $\mathfrak{u}_{P}$
at position $i$ with respect to $b$.
In this case, we define $\mathfrak{l}_{i}:=u_{i}+b(k-i)$. 
\end{enumerate} 
We define a non-negative integer sequence $\mathfrak{u'}_{b}:=(u'_1,\ldots,u'_{n})$ by 
\begin{align*}
u'_{i}:=\#\{ j<i |\ \mathfrak{l}_{j}<\mathfrak{l}_{i}\}.
\end{align*} 
We construct an $m$-tuple of non-negative integer sequences by the following way:
\begin{enumerate}
\item Set $b=m$ and $\mathfrak{u}:=\mathfrak{u}_{P}$.
\item Construct $\mathfrak{u'}_{b}$ as above.
\item Replace $\mathfrak{u}$ by $\mathfrak{u}-\mathfrak{u'}_{b}$, and 
decrease $b$ by one. Then, go to (2). The algorithm stops when $b=1$.
\end{enumerate}

\begin{example}
Let $m=3$ and $P\in\mathfrak{D}_{10}^{(1,3)}$ be a rational Dyck path 
with the step sequence $(0,0,0,2,5,6,17,18,18,20)$.
We have three second primitive integer sequence with respect to $3$:
\begin{align*}
(0,0,0,2,5,6), \quad (0,0,2,5,6), \quad (18,18,20).
\end{align*}
The integer sequence $\mathfrak{l}^{3}$ is given by 
\begin{align*}
\mathfrak{l}^{3}=(6,5,1,2,3,4,7,10,8,9).
\end{align*}
The step sequence $\mathfrak{u'}_{3}=(0,0,0,1,2,3,6,7,7,8)$.
By a similar computation, we have 
\begin{align*}
&\mathfrak{u'}_{2}=(0,0,0,1,2,2,6,6,6,7), \\
&\mathfrak{u'}_{1}=(0,0,0,0,1,1,5,5,5,5).
\end{align*}
\end{example}

\begin{lemma}
\label{lemma:sttostm}
The integer sequences $\mathfrak{u'}_{b}$ for $1\le b\le m$ are 
non-decreasing sequence.
\end{lemma}
\begin{proof}
By definition of $\mathfrak{u'}_{b}$, when we have a sequence 
of distinct integers $(u_{i},u_{i+1},\ldots,u_{j})$, 
we have a strictly increasing integer sequence 
$(\mathfrak{l}_{i},\ldots,\mathfrak{l}_{j})$.
When we have a sequence $(u_{i},\ldots,u_{j})$ such that 
all $u_{k}$'s are the same value, 
we have a strictly decreasing integer sequence 
$(\mathfrak{l}_{i},\ldots,\mathfrak{l}_{j})$.

Suppose that we have two second primitive subsequences 
at positions $i$ and $j$ with $i<j$.
We denote these two subsequence by $(u_i,\ldots,u_{k})$
and $(u_j,\ldots u_{k'})$.
By Eqn. \ref{eqn:2ndps}, it is clear that we have 
$k<j$ or $k'\le k$, and $j<k\le k'$ never happens.

From these observations, it is easy to see that $\mathfrak{u'}_{b}$ 
is a non-decreasing integer sequence.
\end{proof}

From Lemma \ref{lemma:sttostm}, an integer sequence $\mathfrak{u'}_{b}$, 
$1\le b\le n$, gives a step sequence for some Dyck path.
This naturally gives the following definition from a Dyck path 
to an $m$-tuple of Dyck paths.
\begin{defn}
\label{defn:gamma}
We denote by $\gamma$ the map from $\mathfrak{u}_{P}$ to 
$\mathfrak{u'}_{b}$ for $1\le b\le m$.
\end{defn}

The next proposition shows that the step sequences 
$\mathfrak{u'}_{b}$ 
are the same step sequences obtained from the parenthesis presentation.
\begin{prop}
let $\alpha^{I}_{i}(\pi)$ for $1\le i\le m$ be an $m$-tuple of Dyck paths obtained 
from the parenthesis presentation of type $I$.
Then, the step sequence of $\alpha^{I}_{i}(\pi)$ is given by $\mathfrak{u'}_{i}$
constructed as above. 
\end{prop}
\begin{proof}
Since $\mathfrak{u'}_{b}$ for $1\le b\le m$ is defined recursively with 
respect to $b$, it is enough to prove the statement for $b=m$.

Recall that we construct a Dyck path $\alpha^{I}_{m}(\pi)$ from a parenthesis 
presentation $\alpha^{\ast}(\pi)$.
In $\alpha^{\ast}(\pi)$, the marks $\ast$'s are enumerated from right to 
left by $1,\ldots,m,1,\ldots,m,\ldots,m$.
It is clear that the label $\mathfrak{l}^{m}$ constructed from the 
second primitive subsequence is nothing but the position of $\ast$ with 
label $m$ in $\alpha^{\ast}(\pi)$. 
Therefore, the integer sequence $\mathfrak{u'}_{m}$ is exactly the step 
sequence of $\alpha^{I}_{m}(\pi)$.
\end{proof}

\begin{example}
Let $m=3$ and $\mathfrak{u}(P)=(0,3,4)$. 
We have three step sequences: $\mathfrak{u}(\alpha^{I}_1)=(0,1,1)$, 
$\mathfrak{u}(\alpha^{I}_2)=(0,1,1)$ 
and $\mathfrak{u}(\alpha^{I}_3)=(0,1,2)$.
In terms of Dyck words, we have
\begin{align*}
\alpha^{I}_1=()(()), \quad \alpha^{I}_2=()(()), \quad \alpha^{I}_3=()()().
\end{align*}
These three Dyck paths satisfy 
$\mathfrak{u}(\alpha^{I}_3)\le_{Y}\mathfrak{u}(\alpha^{I}_{2})
\le_{Y}\mathfrak{u}(\alpha^{I}_{1})$.
\end{example}

We introduce a notion of admissibility for an $m$-tuple of Dyck paths.
\begin{defn}
\label{defn:adm}
Let $\mathfrak{a}$ be an $m$-tuple of Dyck paths.
We say that $\mathfrak{a}$ is admissible if 
\begin{align}
\mathfrak{a}=\gamma\circ\beta(\mathfrak{a}).
\end{align}
We denote by $\mathfrak{A}_{n}^{m}$ be the set of admissible $m$-tuple of
Dyck paths in $\mathfrak{A'}_{n}^{m}$.
\end{defn}

\begin{example}
Let $m=2$ and $n=3$. We have two double Dyck words
\begin{align*}
\mathfrak{a}_{1}:=(NNNEEE , NENENE), \quad \mathfrak{a}'_{1}:=(NNENEE,NENNEE).
\end{align*}
By applying $\beta$ on both $\mathfrak{a}_1$ and $\mathfrak{a}'_{1}$, 
we have $\mathfrak{u}=(0,1,2)$.
It is easy to see that $\mathfrak{a}_1$ is admissible and $\mathfrak{a}'_1$ is non-admissible.
Similarly, consider two double Dyck words
\begin{align*}
\mathfrak{a}_{2}:=(NNENEE,NNENEE), \quad \mathfrak{a}'_{2}:=(NNNEEE,NNEENE).
\end{align*}
Both words give $\mathfrak{u}=(0,0,2)$ by the action of $\beta$, 
and $\mathfrak{a}'_{2}$ is non-admissible.

We have $14$ double Dyck words in $\mathfrak{A'}_{n}^{m}$, and the above two 
double Dyck words $\mathfrak{a}'_{1}$ and $\mathfrak{a}'_2$ are non-admissible.
\end{example}

Recall that we have a map $\alpha^{I}:\mathfrak{S}_{n}^{(m)}\rightarrow\mathfrak{A'}_{n}^{m}$.
The map $\alpha^{I}$ is not injective as shown in Proposition \ref{prop:312avoid}.
In Definition \ref{defn:adm}, we introduce the notion of admissibility on an $m$-tuple 
of Dyck words. 
This admissibility reflects that there exists a pair of multi-permutations $\pi$ and $\pi'$ 
in $\mathfrak{S}_{n}^{(m)}$ such that $\alpha^{I}(\pi)=\alpha^{I}(\pi')$.
Therefore, the admissibility detects the obstacle for injectivity of $\alpha^{I}$.
Since $\mathfrak{A'}_{n}^{m}$ contains a non-admissible $m$-tuple of Dyck words,
we consider a map $\alpha^{\ast}:\mathfrak{A}^{\ast}_{n,m}\rightarrow\mathfrak{A}_{n}^{m}$

Below, we will construct the inverse  of $\alpha^{\ast}$ 
(denoted by $(\alpha^{\ast})^{-1}$) which maps $\mathfrak{A}_{n}^{m}$ to 
$\mathfrak{A}^{\ast}_{n,m}$.

Let $\mathfrak{a}:=(a_1,\ldots,a_m)\in\mathfrak{A}_{n}^{m}$.
Let $N^{i}_{j}$ be the $j$-th ``$($" from left in $a_{i}$ with $1\le i\le m$ and $1\le j\le n$.
We define a sequence of non-negative integers $s^{N}:=(s^{N}_1,\ldots,s^{N}_{n})$ by
\begin{align}
s^{N}_{j}:=\sum_{i=1}^{m}n^{i}_{j},
\end{align}
where $n^{i}_{j}$ is the number of ``$)$"'s which is left to $N^{i}_{j}$ in $a_{i}$.
Similarly, let $E^{i}_{j}$ be the $j$-th ``$)$" from right in $a_{i}$,
We define a sequence $s^{E}:=(s^{E}_1,\ldots,s^{E}_{n})$ by
\begin{align}
s^{E}_{j}:=\sum_{i=1}^{m}e^{i}_{j},
\end{align}
where $e^{i}_{j}$ is the number of ``$($"'s which is right to $E^{i}_{j}$ in $a_{i}$.

We construct a parenthesis presentation from $s^{N}$ and $s^{E}$ as follows.
\begin{defn}
\label{defn:atopp}
Let $\mathfrak{a}$, $s^{N}$ and $s^{E}$ as above. 
We construct a parenthesis presentation as follows:
\begin{enumerate}
\item We put $nm$ $\ast$'s in line.
\item We insert $n$ ``$($"'s just before the $s^{N}_{i}+1$-th $\ast$ from left in $nm$ $\ast$'s 
for $1\le i\le m$. 
\item We insert $n$ ``$)$"'s just before the $s^{E}_{i}+1$-th $\ast$ from right in $nm$ $\ast$'s 
for $1\le i\le m$.
\end{enumerate}
We denote by $(\alpha^{\ast})^{-1}$ the map from $\mathfrak{A}_{n}^{m}$ 
to $\mathfrak{A}^{\ast}_{n,m}$ defined as above.
\end{defn}

\begin{prop}
\label{prop:atopp}
Let $\mathfrak{a}$ be an $m$-tuple of Dyck paths in $\mathfrak{A}_{n}^{m}$.
Then, the parenthesis presentation $(\alpha^{\ast})^{-1}(\mathfrak{a})$
gives $\mathfrak{a}$ by use of Definition \ref{defn:alphaast}.
\end{prop}
\begin{proof}
Since $\mathfrak{a}\in\mathfrak{A}_{n}^{m}$, the integer sequence 
$\mathfrak{u}=\beta(\mathfrak{a})$ is a step sequence for some 
Dyck path $P$ by Proposition \ref{prop:utoa1}.
Since each word $a_{i}$, $1\le i\le m$, in $\mathfrak{a}$ is a Dyck 
word, the sequence $s^{N}$ constructed from $\mathfrak{a}$ is 
the sum of the step sequences $\mathfrak{u}(a_{i})$.
Recall that we construct a parenthesis presentation from 
$\pi(P)\in\mathfrak{S}_{n}^{(m)}$ for a Dyck path $P$ 
in Definition \ref{defn:alphaast}.
By the correspondence between $\pi(P)$ and a $(b+1)$-ary 
tree, the position of ``$($" in $\alpha^{\ast}(\pi)$ 
is given by the step sequence $\mathfrak{u}(P)$.
This process is realized by the step (2) in Definition \ref{defn:atopp}.

Similarly, the position of ``$)$" in a parenthesis presentation is 
given by the step (3) in Definition \ref{defn:atopp}.

Finally, we have $\gamma(\mathfrak{u})=\mathfrak{a}$ since 
$\mathfrak{a}\in\mathfrak{A}_{n}^{m}$, which implies 
the parenthesis presentation $(\alpha^{\ast})^{-1}(\mathfrak{a})$ 
gives $\mathfrak{a}$ by Definition \ref{defn:alphaast}.
This completes the proof.
\end{proof}

\begin{example}
Let $\mathfrak{a}:=(\alpha^{I}_1,\alpha^{I}_2)$ be two Dyck words:
\begin{align*}
\alpha^{I}_1=((())()), \quad \alpha^{I}_2=()()()().
\end{align*}
We have $s^{N}=(0,1,2,5)$ and $s^{E}=(0,1,3,4)$. Therefore, the parenthesis presentation with 
$\ast$ corresponding to $\mathfrak{a}$ is given by 
$(*(*(**)*)(**)*)$.
It is easy to check that this parenthesis presentation gives $\mathfrak{a}$ by $\alpha^{\ast}$.
\end{example}

\begin{theorem}
We have a natural bijection between $\mathfrak{S}_{n}^{(m)}(312)$ and 
$\mathfrak{A}_{n}^{m}$.
\end{theorem}
\begin{proof}
Suppose that $\alpha^{\ast}\in\mathfrak{A}_{n}^{m}$ be a parenthesis presentation 
in $\mathfrak{A}_{n,m}^{\ast}$.
One can construct a $(m+1)$-ary tree by reversing the algorithm just above 
Definition \ref{defn:alphaast}.
More precisely, we enumerate left parentheses from left to right by 
$1,2\ldots,n$ in $\alpha^{\ast}$.
Since $\alpha^{\ast}$ is balanced, every right parenthesis can be 
enumerated by an integer corresponding to a left parenthesis.
By reading ``$($"'s $\ast$'s and ``$)$"'s from left to right, 
we obtain a $(m+1)$-ary tree $T$ with labels in $[1,n]$.
Then, we obtain a $m$-Stirling permutation $\pi(T)$ from the tree $T$.
By construction of the tree $T$, it is obvious that $\pi(T)$ avoids 
the pattern $312$.
Further, given $\alpha^{\ast}$, we have a unique $m$-Stirling permutation
avoiding pattern $312$.
Conversely, once a $m$-Stirling permutation given, we have a unique 
parenthesis presentation in $\mathfrak{A}_{n,m}^{\ast}$.
Therefore, we have a bijection between $\mathfrak{S}_{n}^{(m)}(312)$
and $\mathfrak{A}_{n,m}^{\ast}$.

Given $\mathfrak{a}\in\mathfrak{A'}_{n}^{m}$, we have a unique 
parenthesis presentation $\alpha^{I}\in\mathfrak{A}_{n,m}^{\ast}$.
However, the reverse is not true in general. 
The admissibility of $\mathfrak{a}$ gives a unique $\mathfrak{a}$
from $\alpha^{I}$  by Proposition \ref{prop:atopp}.
From these observations, we have a bijection between 
$\mathfrak{A}_{n,m}^{\ast}$ and $\mathfrak{A}_{n}^{m}$.
By the observations in the previous paragraph,
we have a natural bijection between $\mathfrak{S}_{n}^{(m)}(312)$
and $\mathfrak{A}_{n}^{m}$.
\end{proof}

\subsection{Rotation on a Dyck path}
\label{sec:rotDp}
We define a rotation on a Dyck path in $\mathfrak{D}_{n}^{(1,1)}$.
This rotation is a generalization of the rotation of a Dyck path 
defined in Section \ref{sec:rot}

Let $q:=(q_1,q_2,\ldots,q_{2n})\in\{N,E\}^{2n}$ be a Dyck path of size $n$.
\begin{defn}
\label{defn:rotDp}
Suppose $q_{i}=E$ and $q_{i+1}=N$ for some $i\in[2,2n]$. 
The {\it rotation} of size $m$ at position $i$ defines a Dyck path $q':=(q'_1,\ldots,q'_{2n})$
by 
\begin{align}
&q'_{j}=q_{j},\quad \forall j\notin [i,i+2m], \\
&q'_{j}=q_{j+1}, \quad \forall j\in[i,i+2m-1], \\
&q'_{i+2m}=q_{i}.
\end{align}
\end{defn}

\begin{defn}
\label{defn:adrotDp}
Let $q:=(q_1,q_2,\ldots,q_{2n})\in\{N,E\}^{2n}$ be a Dyck path.
We say that the rotation of size $m$ at $q_{i}$ is {\it admissible} if 
the partial path $q'=(q_{i+1},\ldots,q_{i+2m})$ is a Dyck word.
We say that a Dyck word is irreducible if it can not be expressed 
as a concatenation of two Dyck words.
Similarly, if $q'$ is irreducible, we say the rotation is irreducible.
\end{defn}

Definition \ref{defn:rotDp} is equivalent to the definition of the rotation on a 
Dyck path given in Section \ref{sec:rot} when the rotation is irreducible.

\subsection{\texorpdfstring{Rotation order on an $m$-tuple of Dyck words}
{Rotation order on an m-tuple of Dyck words}}
\label{sec:rotmDw}
Let $\mathfrak{a}:=(a_1,\ldots,a_m)$ be an $m$-tuple of Dyck words in $\mathfrak{A}_{n}^{m}$, 
and $a_{i}=(a_{i,1},\ldots,a_{i,2n})\in\{N,E\}^{2n}$ for $1\le i\le m$ be Dyck words.

Let $\alpha^{\ast}(P)$ be a parenthesis presentation of $P$, and $\mathfrak{a}$ 
be its $m$-tuple of Dyck words. 
Let $\epsilon(i)$ be the $i$-th left parenthesis ``$($" in $\alpha^{\ast}(P)$ from left.
Recall we enumerate $\ast$'s by integers in $[1,m]$ in 
Definition \ref{defn:alphaast}.

Recall that we construct $\alpha^{\ast}(P)$ from a $(m+1)$-ary tree. 
We say that an $\ast$ belongs to the left parenthesis $\epsilon(i)$ if 
the $\ast$ corresponds to the region between edges labeled by $i$ in 
the $(m+1)$-ary tree.
Let $n_{1}$ be the position of the integer $1$ 
belonging to the left parenthesis $\epsilon(i)$.
We denote by $n'$ the number of left parentheses between $\epsilon(i)$ and $n_1$, which
includes $\epsilon(i)$ as well.

Suppose that we have an $\ast$ left next to $\epsilon(i)$ in $\alpha^{\ast}(P)$.
Then, we have $n'$ left parentheses between $\epsilon(i)$ and $n_1$ and 
$n'$ right parentheses corresponding to these $n'$ left parentheses. 
We move these left and right parentheses left by one step 
in $\alpha^{\ast}(P)$ as follows.
First, we delete the remaining $n-n'$ balanced parentheses from $\alpha^{\ast}(P)$.
Secondly, if a parenthesis $p$ is right to the $x$-th $\ast$, then we move 
$p$ right to the $x-1$-th $\ast$.
Thirdly, we add $n-n'$ balanced parentheses such that the positions of them 
are the same as before the deletion.
For example, if we move the right-most ``$)$" in $**(***))$ left by one step, 
we obtain a sequence $*)*(***)$ since $(***)$ is balanced.

\begin{defn}
Let $\epsilon(i)$ and $n_1$ be defined as above.
We call the operation on $\alpha^{\ast}(P)$ characterized by $\epsilon(i)$ and $n_1$ 
a rotation on the parenthesis presentation at $i$.
\end{defn}

\begin{prop}
\label{prop:rotonpp}
Let $\alpha^{\ast}(P')$ be a parenthesis presentation obtained 
by a rotation on $\alpha^{\ast}(P)$ at $i$.
The rotation on the parenthesis presentation coincides with the rotation 
on the Dyck path $P$.
In other words, $P'$ is the Dyck path obtained from $P$ by a rotation at position $i$.
\end{prop}
\begin{proof}
We show that the rotation on $\alpha^{\ast}(P)$ gives the primitive 
subsequence of the step sequence $\mathfrak{u}_{P}:=(u_1,\ldots,u_{n})$ at position $i$.
The condition that we have an $\ast$ left to $\epsilon(i)$ implies 
$u_{i-1}<u_{i}$ in the step sequence.

The parenthesis presentation $\alpha^{\ast}(P)$ gives a rational 
Dyck path in $\mathfrak{D}_{n}^{(1,m)}$ by deleting all right 
parentheses and by replacing ``$($" and $\ast$ by $N$ and $E$ respectively. 
From Proposition \ref{prop:rot2}, it is enough to show that 
we have the same horizontal distance at $\epsilon(i)$ and $n_1$.
However, this is obvious since $n_1$ is labeled by $1$, we have 
$n'$ left parentheses and $mn'$ $\ast$'s. 

To move $n'$ left parentheses and $n'$ right parentheses left by one step
corresponds to Eqn. (\ref{eqn:defrot}) for the primitive subsequence at 
position $i$.

From these observations, the rotation on a parenthesis presentation is 
equivalent to the rotation on the Dyck path $P$.
\end{proof}

\begin{example}
We consider a Dyck path $P\in\mathfrak{D}_{4}^{(1,3)}$ with the step 
sequence $(0,1,2,4)$.
The parenthesis presentation of $P$ is $\alpha^{\ast}:=(*(*(**(***)*)**)**)$.
When $i=2$, we have three left parentheses ``$($"'s after the second 
left parenthesis and before $n_1$.
We move three left parentheses and three right parentheses in $\alpha^{\ast}$.
Then, this operation results in $((*(**(***)*)**)***)$. 
This new parenthesis presentation corresponds to a Dyck path 
with the step sequence $(0,0,1,3)$.
\end{example}

Let $\mathfrak{a}$ be the $m$-tuple of Dyck words corresponding to a rational 
Dyck path $P\in\mathfrak{D}_{n}^{(1,m)}$ with step sequence $\mathfrak{u}_{P}$.
Let $l_{\ast}$ be the label of the $\ast$ left to $\epsilon(i)$, 
and $r$ be the number of $\ast$'s with label $l_{\ast}$ left to $\epsilon(i)$.
\begin{defn}
Let $n'$ and $l_{\ast}$ defined as above.
We perform a rotation of size $n'$ at position $r$ on $a_{l_{\ast}}$,  
and the other $a_{p}$'s with $p\neq l_{\ast}$ remain the same.
We call this rotation a rotation of $\mathfrak{a}$.  
\end{defn}

\begin{prop}
\label{prop:rotonmDw}
The rotation on $\mathfrak{a}$ is admissible.
\end{prop}
\begin{proof}
We move $n'$ left parentheses and $n'$ right parentheses 
left by one step.
Note that the labels on $\ast$'s which belong to $n'$ left 
parentheses are not changed by the rotation.
However, the $\ast$ with label $l_{\ast}$ is moved to
right by $mn'$ step in $\alpha^{\ast}$. 
In terms of a Dyck word $a_{l_{\ast}}$, the $r$-th right parenthesis 
is moved to right by $2n'$ step.
Since the label of $n_1$ is one, we have $n'$ left parentheses 
and $n'$ $\ast$'s with label $l_{\ast}$. 
By construction of an $m$-tuple of Dyck words, we have a Dyck word of size 
$n'$ after the $r$-th right parenthesis in $a_{l_{\ast}}$.
From these, the rotation on $\mathfrak{a}$ is admissible.
\end{proof}

\begin{remark}
The rotation on $\mathfrak{a}$ may not be irreducible.
For example, we consider a Dyck path in $\mathfrak{D}_{4}^{(1,3)}$ 
with step sequence $(0,1,2,7)$. The triple Dyck words 
$\mathfrak{a}=(\alpha^{I}_1,\alpha^{I}_2,\alpha^{I}_3)$
are given by 
\begin{align*}
\alpha^{I}_1=((())()), \quad
\alpha^{I}_2=((())()), \quad
\alpha^{I}_3=()()()().
\end{align*}
The rotation at $i=2$ gives the rotation of size $2$ on $\alpha^{I}_3$. 
Thus, the Dyck word $\alpha^{I}_3$ is rotated as
\begin{align*}
\alpha^{I}_3=()()()()\rightarrow (()())().
\end{align*}
Note that this rotation is not irreducible.
\end{remark}

\begin{theorem}
The rotation on $\mathfrak{a}$ is equivalent to the rotation on $\mathfrak{u}_{P}$
defined in Definition \ref{sec:rot}.
\end{theorem}
\begin{proof}
From Proposition \ref{prop:rotonpp}, it is enough to show that 
the rotation on $\mathfrak{a}$ is equivalent to the rotation on 
a corresponding parenthesis presentation.
This is obvious form the proof of Proposition \ref{prop:rotonmDw}.
This completes the proof.
\end{proof}

\section{The strip-decompositions and parenthesis presentations}
\label{sec:sdpp}
\subsection{Parenthesis presentations and step sequences}
Let $P\in\mathfrak{D}_{n}^{(a,b)}$ be a rational Dyck path, and 
$\mathfrak{a}:=(a_1,\ldots,a_b)\in(\mathfrak{D}_{an}^{(1,1)})^{b}$ be a $b$-tuple of 
Dyck paths in the parenthesis presentation of type $I$ obtained from $P$ 
by Definition \ref{defn:alphaast}.

We construct a Young diagram from $\mathfrak{a}:=(a_1,\ldots,a_b)$ in the following way.
We replace ``$($" by $N$ and ``$)$" by $E$ in $a_{i}$ and denote 
by $b_{i}:=b_{i,1}\ldots b_{i,2an}\in\{N,E\}^{2an}$ the sequence of $N$'s and $E$'s 
obtained from $a_{i}$.
Then, we construct a sequence $A(P):=A_1A_2\ldots A_{2abn}$ 
of $N$'s and $E$'s of size $2abn$ by
\begin{align*}
A_{i}:=b_{b+1-p,q}
\end{align*}
where $p$ and $q$ be uniquely determined by $i=b(q-1)+p$ 
with  $1\le p\le b$ and $1\le q\le 2an$. 

Since $\mathfrak{a}$ is a set of Dyck paths, and by construction of $A$,
the sequence $\mathfrak{A}$ is also a Dyck path of size $2abn$.
Then, we define a Young diagram $Y(P)$ surrounded by two Dyck paths $A$ 
and $N^{abn}E^{abn}$.
Here, we replace $N$ (resp. $E$) by a vertical (resp. horizontal) line of length one.

Let $b=(i,j)$ be a box in the Young diagram in $Y(P)$.
The content $c(i,j)$ of the box $(i,j)$ is defined as one plus the sum 
of two statistics $\mathrm{arm}(b)$ and $\mathrm{leg}(b)$.
Here, $\mathrm{arm}(b)$ is the number of boxes in $Y(P)$ right to $b$, 
and $\mathrm{leg}(b)$ is the number of boxes in $Y(P)$ below $b$. 

We put a circle on the box $b$ such that $c(i,j)\equiv0\mod{b}$ in $Y(P)$.
We count the number of circled boxes in each column in $Y(P)$ and construct 
a sequence $v(P):=(v_1,v_2,\ldots,v_{l'})$ of non-negative integers.
We delete zeros from $v(P)$ and obtain a partition denoted by $v(P)$ by abuse
of notation.

\begin{prop}
\label{prop:YDu}
Let $v(P)$ be a partition obtained from $P$ as above.
Then, the step sequence $\mathfrak{u}_{P}$ is given by 
\begin{align*}
\mathfrak{u}_{P}=v(P)^{t}
\end{align*}
where $v(P)^{t}$ is the transposition of $v(P)$.
\end{prop}

Before proceeding to the proof of Proposition \ref{prop:YDu}, 
we introduce the following lemma.

\begin{lemma}
\label{lemma:YDcont}
Let $Y$ be a Young diagram with $m$ rows and $n$ columns.
The content of the left-top box is $n+m-1$.
\end{lemma}
\begin{proof}
Recall that the content of a box $b$ is one plus the sum 
of two statistics $\mathrm{arm}(b)$ and $\mathrm{leg}(b)$.
The left-top box has $n-1$ boxes right to it and $m-1$ boxes 
below it. Therefore, the content is $(n-1)+(m-1)+1=n+m-1$. 
\end{proof}

\begin{proof}[Proof of Proposition \ref{prop:YDu}]
The south-east boundary of $Y(P)$ consists of $N$'s and $E$'s, 
which come from $N$ and $E$ in $a_{i}$, $1\le i\le b$.
Suppose we have a pair of $E$ and $N$  such that $N$ is right to 
this $E$.
We denote by $e_{E}$ and $e_{N}$ the boundary edges in $Y(P)$ 
which corresponds to this pair of $E$ and $N$ in $a_{i}$. 
Let $b$ be the box in the same column as $e_{E}$ and 
in the same row as $e_{N}$. 
Since the boundary of $Y(P)$ is obtained from $\mathfrak{a}$ by 
reading the entries of $a_{i}$ in order, the number of edges 
between $e_{E}$ and $e_{N}$ in $Y(P)$ is $kb-1$, where $k$ 
is the distance between $E$ and $N$ in $a_{i}$.
By applying Lemma \ref{lemma:YDcont} to the box $b$, we have 
that the content of $b$ is $kb$, which is zero modulo $b$. 
All boxes with content zero modulo $b$ are obtained as above.

The step sequence of $P$ is given by the sum of the step sequences
of $a_{i}$'s.
This implies that a box in the Young diagram characterized 
by $a_{i}$ is one-to-one correspondence to a box in $Y(P)$.
By combining this with the argument in the previous paragraph,
the partition $v(P)$ gives the step sequence $\mathfrak{u}_{P}$ 
by taking its transposition. 
\end{proof}

\begin{example}
We consider the same rational Dyck path in $\mathfrak{D}_{2}^{(2,3)}$ as Figure \ref{fig:rDyck}.
Since $b=3$, we have triple Dyck paths $\mathfrak{a}:=(a_{1},a_{2},a_{3})$ by the parenthesis
presentation of type $I$, where
\begin{align}
a_{1}=(((()))),\quad a_2=((()())),\quad a_3=()()()().
\end{align}
Then, we have a Dyck path of size $12$ which is 
\begin{align}
A(P)=N^{3}EN^{5}E^{2}N^{3}E^{4}NE^{5}.
\end{align}
This gives a Young diagram $\lambda=(7,3^{3},1^{5})$ depicted as  
\begin{align*}
\tikzpic{-0.5}{[scale=0.6]
\draw(0,0)--(7,0)(0,-1)--(7,-1)(0,-2)--(3,-2)(0,-3)--(3,-3)(0,-4)--(3,-4)
     (0,-5)--(1,-5)(0,-6)--(1,-6)(0,-7)--(1,-7)(0,-8)--(1,-8)(0,-9)--(1,-9);
\draw(0,0)--(0,-9)(1,0)--(1,-9)(2,0)--(2,-4)(3,0)--(3,-4)(4,0)--(4,-1)(5,0)--(5,-1)
     (6,0)--(6,-1)(7,0)--(7,-1);
\draw (0.41,-0.5)node{\scalebox{0.9}{\circnum{15}}}(1.45,-0.5)node{\scalebox{0.9}{\circnum{9}}}(2.45,-0.5)node{$8$}
      (3.45,-0.5)node{$4$}(4.45,-0.5)node{\scalebox{0.9}{\circnum{3}}}(5.45,-0.5)node{$2$}(6.45,-0.5)node{$1$};
\draw(0.45,-1.5)node{$10$}(1.45,-1.5)node{$4$}(2.45,-1.5)node{\scalebox{0.9}{\circnum{3}}};
\draw(0.45,-2.5)node{\scalebox{0.9}{\circnum{9}}}(1.45,-2.5)node{\scalebox{0.9}{\circnum{3}}}(2.45,-2.5)node{$2$};
\draw(0.45,-3.5)node{$8$}(1.45,-3.5)node{$2$}(2.45,-3.5)node{$1$};
\draw(0.45,-4.5)node{$5$}(0.45,-5.5)node{$4$}(0.45,-6.5)node{\scalebox{0.9}{\circnum{3}}}
(0.45,-7.5)node{$2$}(0.45,-8.5)node{$1$};
}.
\end{align*}
We have $v(P)=(3,2,1,1)$, and we have a partition $(4,2,1)$ 
by the transposition of $v(P)$.
Note $(4,2,1)$ is equal to $\mathfrak{u}_{P}$ by ignoring $0$.
\end{example}

\subsection{\texorpdfstring{A $b$-tuple of Dyck words of type $II$}{a b-tuple of Dyck words of type II}}
\label{sec:mockDw}
Let $P$ be a rational Dyck path in $\mathfrak{D}_{n}^{(a,b)}$.
In Definition \ref{defn:alphaast}, we define the $m$-tuple of Dyck words from 
an $m$-Stirling permutation by the parenthesis presentation of type $I$.

By the map $\zeta$ in Definition \ref{defn:zeta}, we have a $b$-Stirling 
permutation from $P$.
We give an another different definition of a $b$-tuple of Dyck words.

\begin{defn}
Let $\pi\in\mathfrak{S}_{n}^{(b)}$ and $\alpha^{\ast}(\pi)$ be a parenthesis 
representation with $\ast$.
We enumerate $\ast$ from right to left by $1,\ldots,b,1,\dots,b,1,\ldots,b$ by ignoring 
the parentheses $($'s and $)$'s.
By the same processes (1) and (2) in Definition \ref{defn:alphaast}, 
we have a $b$-tuple of Dyck words.
We denote by $\alpha^{II}$ the map from $\mathfrak{S}_{n}^{(b)}$ 
to $\mathfrak{U'}_{n}^{b}$, and call it 
the parenthesis presentation of type $II$.
\end{defn}

Let $\mathfrak{u}_{P}$ be the step sequence associated with $P$.
Suppose that $P$ is expressed in terms of steps $N$'s and $E$'s 
by $P=N^{i_1}E^{j_1}\ldots N^{i_r}E^{j_r}$ for some integer $r$.
We define a rational Dyck path $\overline{P}:=(\overline{P}_1,\ldots,\overline{P}_{2abn})$
by 
\begin{align*}
\overline{P}:=N^{bi_1}E^{j_1}N^{bi_2}E^{j_2}\ldots N^{bi_r}E^{j_r}E^{i_{r+1}},
\end{align*}
where the non-negative integer $i_{r+1}$ is given by 
\begin{align*}
i_{r+1}:=b(i_1+i_2+\ldots+i_{r})-(j_1+j_2+\ldots+j_{r}).
\end{align*}
We construct a $b$-tuple of Dyck words 
$\overline{\mathfrak{p}_{i}}:=(\overline{p_{i,1}},\ldots,\overline{p_{i,an}})$ by
\begin{align}
\label{eqn:Ptomock}
\overline{p_{i,j}}:=\overline{P}_{b(j-1)+b+1-i}
\end{align}
for $1\le i\le b$ and $1\le j\le an$.

\begin{prop}
\label{prop:mock1}
Let $\alpha^{II}$ and $\{\overline{\mathfrak{p}_{i}}\}_{i=1}^{b}$ be constructed 
from $\pi\in\mathfrak{S}_{n}^{(a,b)}$ as above. 
Then we have 
\begin{align}
\alpha^{II}(\pi)
=(\overline{\mathfrak{p}_{1}},\overline{\mathfrak{p}_{2}},\ldots,\overline{\mathfrak{p}_{b}}).
\end{align}
\end{prop}
\begin{proof}
We enumerate all steps $N$'s and $E$'s in $\overline{P}$ by $(1,\ldots,b)^{2an}$ 
from right to left.
Then, the $i$-th step $E$ from right is enumerated by $i\mod{b}$ 
since the numbers of successive $N$'s between two $E$'s are 
zero modulo $b$ In a Dyck path $\overline{P}$. 

Recall we construct $\alpha^{II}(\pi)$ from a parenthesis 
presentation $\alpha^{\ast}(\pi)$.
We enumerate the $i$-th $\ast$ in $\alpha^{\ast}(\pi)$ by $i\mod{b}$ from 
right to left. 
Since an $\ast$ in $\alpha^{\ast}(\pi)$ corresponds to an $E$ in $\overline{P}$, 
this enumeration for $\ast$'s is compatible with the enumeration for $E$.
Therefore, Eqn. (\ref{eqn:Ptomock}) yields the $b$-tuple of Dyck words 
$\alpha^{II}(\pi)$. 
\end{proof}

\begin{defn}
\label{defn:mock}
We construct a $b$-tuple of step sequences $\{\mathfrak{u}(\mathfrak{p_i})\}_{i=1}^{b}$ from $P$
by the following processes. 
\begin{enumerate}
\item Set $i=n$ and $\mathfrak{u}:=\mathfrak{u}_{P}$.
\item We define 
\begin{align}
\label{eqn:mockb2}
\mathfrak{u}(\mathfrak{p_{i}})
:=\left\lceil\genfrac{}{}{0.8pt}{}{\mathfrak{u}}{i}\right\rceil.
\end{align}
\item Replace $\mathfrak{u}$ by $\mathfrak{u}-\mathfrak{u}(\mathfrak{p_{i}})$, 
and decrease $i$ by one. Go to (2). The algorithm stops when $i=1$.
\end{enumerate}
We denote by $\gamma^{II}$ the map from $\mathfrak{U}_{n}^{b}$ to $\mathfrak{A'}_{n}^{b}$.
\end{defn}

The following lemma is a direct consequence of Definition \ref{defn:mock}.
\begin{lemma}
\label{lemma:mocksum}
The step sequence $\mathfrak{u}_{P}$ is given by 
\begin{align*}
\mathfrak{u}_{P}=\sum_{i=1}^{b}\mathfrak{u}(\mathfrak{p}_{i}).
\end{align*}
\end{lemma}

\begin{prop}
\label{prop:alphgamma}
Let $\pi$ be a $b$-Stirling permutation for a rational Dyck path, 
and $\alpha^{II}$ and $\gamma^{II}$ be the maps defined as above.
Then, we have $\alpha^{II}(\pi)=\gamma^{II}(\pi)$.
\end{prop}
\begin{proof}
Let $\mathfrak{u}_{P}:=(u_1,\ldots,u_{an})$ be the step sequence of $P$.
From Proposition \ref{prop:mock1}, it is enough to prove 
$\overline{\mathfrak{p}_{i}}=\mathfrak{p}_{i}$ for $1\le i\le b$. 

The Young diagram characterized by $\overline{P}$ is 
$Y:=(u_1^{b},u_{2}^{b},\ldots,u_{an}^{b})$.
As in the proof of Proposition \ref{prop:mock1}, we enumerate 
the $i$-th $E$ from right in $\overline{P}$ by $i\mod{b}$.
As in Proposition \ref{prop:YDu}, if the content of a box in $Y$ 
is zero modulo $b$, the box corresponds to a box in the Young 
diagram for a Dyck path $\mathfrak{p}_{i}$ for some $i$.
From these two observations, it is clear that 
Eqn. (\ref{eqn:mockb2}) gives 
the same Dyck path as $\overline{\mathfrak{p}_{i}}$ for $1\le i\le b$.
\end{proof}

\subsection{From \texorpdfstring{$\alpha^{II}(\pi)$}{a II(pi)} to 
\texorpdfstring{$\alpha^{I}(\pi)$}{aI(pi)}}
In this subsection, we will show a relation between 
$\alpha^{I}(\pi)$ and $\alpha^{II}(\pi)$
by referring to the parenthesis presentation $\alpha^{\ast}(\pi)$.
Before considering a general case, we first consider $b=2$ case.

Let $\pi\in\mathfrak{D}_{n}^{(a,2)}$ and 
$\alpha^{\ast}(\pi)$ be a parenthesis representation with $\ast$.
To obtain $\alpha^{I}(\pi)$ and $\alpha^{II}(\pi)$, we enumerate 
the $\ast$'s by integers $1$ and $2$.
Let $s^{I}$ (resp. $s^{II}$) be a sequence of integers in $[1,2]$ obtained 
from $\alpha^{\ast}(\pi)$ by taking the labels of $\ast$'s to construct 
$\alpha^{I}(\pi)$ (resp. $\alpha^{II}(\pi)$).
Suppose that $s^{I}:=s^{I}_{1}\ldots s^{I}_{2an}\in\{1,2\}^{2an}$ 
and $s^{II}:=s^{II}_{1}\ldots s^{II}_{2an}\in\{1,2\}^{2an}$.
We say that $s^{II}_{i}$ and $s^{II}_{j}$ are a pair if 
$\{s^{II}_{i},s^{II}_{j}\}=\{1,2\}$ and the labels $s^{II}_{i}$ and $s^{II}_{j}$
are obtained from the same integer in $\pi$.
Let $p$ be an integer such that 
\begin{align}
\begin{aligned}
\label{eqn:conds}
&s^{II}_{i}=s^{I}_{i}, \quad \forall i\in[1,p-1], \\
&s^{II}_{p}\neq s^{I}_{p}.
\end{aligned}
\end{align}
Let $q$ be the integer such that $s^{II}_{p}$ and $s^{II}_{q}$ are a pair.
Suppose that $s^{II}_{p}=l$ and $s^{II}_{q}=3-l$ where $l$ is either $1$ or $2$.
We have a double Dyck paths $\alpha^{II}(\pi)=(\mathfrak{p}_1,\mathfrak{p}_2)$ 
expressed in terms of parentheses.
Let $x$ (resp. $y$) be an integer such that the $x$-th (resp. $y$-th) parenthesis 
in $\mathfrak{p}_{l}$ (resp. $\mathfrak{p}_{3-l}$) corresponds to $s^{II}_{p}$ 
(resp. $s^{II}_{q}$).
Note that we have $x\le y$ by construction of $\alpha^{\ast}(\pi)$.
We move the $x$-th parenthesis in $\mathfrak{p}_1$  
just before the $y$-th parenthesis in $\mathfrak{p}_1$ 
if $x<y$ and we do nothing if $x=y$.
Similarly, we move the $y$-th parenthesis in $\mathfrak{p}_2$ just after the 
$x$-th parenthesis in $\mathfrak{p}_2$ if $x<y$ and we do nothing otherwise.

We define new $s^{II}$ as the sequence obtained from $s^{II}$ by switching $s^{II}_{p}$
and $s^{II}_{q}$.
Then, we continue this process until we have $s^{II}=s^{I}$.

We denote by $\alpha''(\pi)$ the parenthesis presentation obtained from 
$\alpha^{II}(\pi)$ by the above procedures.

\begin{prop}
\label{prop:IfromII}
We have $\alpha^{I}(\pi)=\alpha''(\pi)$.
\end{prop}
\begin{proof}
The value $p$ is the first position in $s^{II}$ and $s^{II}$ such that 
two sequences have different entries.
Since the value $q$ is a pair of $s_{p}^{II}$, the marks $\ast$  
at $p$-th and $q$-th positions in $\alpha^{\ast}(\pi)$ belong 
to the same label $l$ in $\pi$.
The parentheses and $\ast$'s between the $p$-th and $q$-th positions
in $\alpha^{\ast}$ belong to a label different from $l$.
This implies that switching of the labels $s^{II}_{p}$ and $s^{II}_{q}$
means that we move parentheses at $x$-th position in $\mathfrak{p}_{1}$ right 
and at $y$-th position in $\mathfrak{p}_{2}$ left by keeping other parentheses
unchanged.
By repeating this process, we obtain $\alpha^{I}(\pi)$ from $\alpha^{II}(\pi)$,
which implies $\alpha^{I}(\pi)=\alpha''(\pi)$. 
\end{proof}

We are ready to construct $\alpha^{I}(\pi)$ from $\alpha^{II}(\pi)$
for general $b$.
As in the case of $b$, we define $s^{I}$ and $s^{II}$ as a sequence of 
integers in $[1,b]$ obtained from $\alpha^{\ast}(\pi)$.
We define integers $p$ as conditions (\ref{eqn:conds}).
The integer $q$ is chosen such that $s_{q}^{II}=s_{p}^{I}$ and $s^{II}_{p}$ 
and $s^{II}_{q}$ are a pair.
Suppose $s^{II}_{p}=l_1$ and $s^{II}_{q}=l_2$.
We have a $b$-tuple of Dyck paths $\alpha^{II}(\pi)=(\mathfrak{p}_1,\ldots,\mathfrak{p}_b)$.
Then, we apply the same algorithm as $b=2$ case to $\mathfrak{p}_{l_1}$ and $\mathfrak{p}_{l_2}$.
We continue this process until we have $s^{I}$ from $s^{II}$.
We denote by $\alpha'(\pi)$ the parenthesis presentation obtained from 
$\alpha^{II}(\pi)$ by the above procedures.

The following is obvious from Proposition \ref{prop:IfromII}:
\begin{prop}
\label{prop:ppmock}
We have $\alpha^{I}(\pi)=\alpha'(\pi)$.
\end{prop}

\begin{example}
We consider the path with step sequence $(0,1,1,3)$ in $\mathfrak{D}_{2}^{(2,3)}$.
We have 
\begin{align*}
s^{I}=(3((32(321)1)321)21), \quad s^{II}=(3((21(321)3)213)21),
\end{align*}
where $\ast$'s are replaced by their labels.
One can obtain $s^{I}$ from $s^{II}$ by switching underlined integers 
as follows:
\begin{align*}
(3((\underline{2}1(321)\underline{3})213)21)
&\rightarrow(3((3\underline{1}(321)\underline{2}(213)21)
\rightarrow(3((32(321)1)\underline{2}1\underline{3})21) \\
&\rightarrow(3((32(321)1)3\underline{1}\underline{2})21)
\rightarrow(3((32(321)1)321)21)
\end{align*}
In terms of parenthesis presentation, we have the following 
sequence of triple Dyck paths:
\begin{align*}
\alpha^{II}(\pi)=
\begin{matrix}
((()()))\\
(((\underline{)}())) \\
()((()\underline{)})
\end{matrix}
\rightarrow 
\begin{matrix}
(((\underline{)}()))\\
(((()\underline{)})) \\
()(()())
\end{matrix}
\rightarrow
\begin{matrix}
(((())))\\
((()()\underline{)}) \\
()(()()\underline{)}
\end{matrix}
\rightarrow
\begin{matrix}
(((())))\\
((()()\underline{)}) \\
()(()(\underline{)})
\end{matrix}
\rightarrow
\begin{matrix}
(((())))\\
((()())) \\
()(()())
\end{matrix}
=\alpha^{I}(\pi).
\end{align*}
\end{example}

\section{The horizontal strip-decomposition and orders}
\label{sec:Hsdmp}
\subsection{Horizontal strip-decomposition and multi-permutations}
In this subsection, we study the relation between a rational Dyck path 
$P\in\mathfrak{D}_{n}^{(a,b)}$ and a $b$-Stirling permutation.	
We also show that a $b$-tuple of Dyck paths $\mathbf{v}_{i}$, $1\le i\le b$, 
characterizes the $b$-Stirling permutation.
The $b$-tuple of Dyck paths $\delta(P)$ obtained from the horizontal strip 
decomposition of $P$ (see Definition \ref{defn:stripdec}) 
coincide with the Dyck paths $\mathfrak{v}_{i}$.

\begin{defn}
Let $P_0$ (resp. $P_1$) be the lowest (resp. highest) path associated with $\mathfrak{u}_{P}$ 
in $\mathfrak{D}_{n}^{(a,b)}$.
We define the $b$-Stirling permutations $\pi_{0}$ and $\pi_{1}$ by $\pi_{0}:=\zeta(\mathfrak{u}_{P_0})$ and
by $\pi_{1}:=\zeta(\mathfrak{u}_{P_1})$ respectively.
\end{defn}

Let $g$ be a positive rational number, and 
$\mathfrak{U}_{n}$ be the set of sequences of non-negative integer of length 
$n$.
\begin{defn}
Given a sequence $\mathfrak{u}:=(u_1,\ldots,u_{n})$ of non-negative integers,
we define the map $\zeta_{g}:\mathfrak{U}_{n}\rightarrow\mathfrak{U}_{n}, \mathfrak{u}\mapsto\mathfrak{u'}$ 
by 
\begin{align}
\zeta_{g}(\mathfrak{u}):=\zeta(g\cdot \mathfrak{u}), 
\end{align}
where $g\cdot \mathfrak{u}$ is $(gu_1,\ldots,gu_{n})$.
This definition is well-defined when 
$0\le gu_{i}\le ib$ is a non-negative integers for all $i\in[1,n]$.
\end{defn}

We define $\zeta_{g}$ for a positive rational number, but we use in practice 
$g=a$ or $1/a$.
By definition, it is obvious that $\zeta_{1}$ is $\zeta$ and 
the composition of $\zeta_{g}$ and $\zeta_{1/g}$ is nothing but the identity.

Let $P\in\mathfrak{D}_{n}^{(a,b)}$ be a rational Dyck path
associated with step sequence $\mathfrak{u}_{P}$ and $\delta(P):=(q_1,\ldots,q_{b})$
be its strip-decomposition.
Each Dyck path $q_{i}$ with $1\le i\le b$ is a Dyck path of size $an$.

Let $Q\in\{P_0,P_1\}$ be a rational Dyck path.
We recursively define $b$ sequences $\mathbf{v}_{i}$ with $1\le i\le b$ of $an$ non-negative integers 
from $\mathfrak{u}_{P}:=(u_1,\ldots,u_{an})$ as follows.
First, we define 
\begin{align*}
\mathfrak{u'}_{P}=(u'_1,\ldots,u'_{an}):=
\begin{cases}
\mathfrak{u}_{Q}-\mathfrak{u}_{P}, & Q=P_{0}, \\
\mathfrak{u}_{P}, & Q=P_1.
\end{cases}
\end{align*}
Secondly, we define 
\begin{align}
\label{eqn:hsdecv1}
\begin{aligned}
\mathbf{v}_1&:=\left\lceil \genfrac{}{}{0.8pt}{}{a}{b}\cdot \mathfrak{u'}_{P} \right\rceil, \\
&=(\lceil au'_{1}/b\rceil,\ldots,\lceil au'_{an}/b\rceil).
\end{aligned}
\end{align}
Similarly, for $i\ge2$ we define
\begin{align}
\label{eqn:hsdecv2}
\mathbf{v}_{i}:=\left\lceil(a\cdot\mathfrak{u'}_{P}-\sum_{j=1}^{i-1}\mathbf{v}_{j})/(b-i+1)\right\rceil.
\end{align}

Let $v:=(v_1,\ldots,v_{n})$ be a sequence of $n$ non-negative integers 
such that $v_{i}\le i-1$.
We construct a permutation $w:=(w_1,\ldots,w_n)$ on $[1,n]$ from 
a non-decreasing sequence $\mathfrak{u}:=(u_1,\ldots,u_{n})$ recursively as follows.
\begin{enumerate}
\item Define $S=\{1,\ldots,n\}$ and set $i=n$.
\item Define $w_i$ is the $u_{i}+1$-th smallest element in $S$.
We replace $S$ by $S\setminus\{w_{i}\}$.
\item Decrease $i$ by one and go to (2). Algorithm stops when one has 
a permutation $w$.
\end{enumerate}
We denote by $\eta: v\mapsto w$ the map defined as above.

\begin{defn}
\label{defn:weta}
We have $b$ permutations $\mathbf{w}_{i}$ with $1\le i\le b$ 
from $P\in\mathfrak{D}_{n}^{(a,b)}$ by
\begin{align*}
\mathbf{w}_{i}:=\eta(\mathbf{v}_{i}).
\end{align*}
\end{defn}

In case of $Q=P_1$, we give another description of $\mathbf{w}_{i}$ in terms of a rooted tree 
defined in Section \ref{sec:rtree}.
Let $T(q)$ be a rooted tree for the Dyck path $q\in\mathfrak{D}_{an}^{(1,1)}$.
We put labels on the edge in $T(q)$ by the post-order, and denote by $L$ this 
decreasing label.
Then, we obtain a permutation $w(L)$ by reading the labels in $L$ by the pre-order.	

We denote by $L_{i}$ with $1\le i\le b$ the decreasing labels of the rooted 
tree $T(q_{i})$ with $1\le i\le b$ and 
by $w(L_{i})$ the pre-order word of $L_{i}$.

\begin{prop}
\label{prop:wrt}
Let $w(L_{i})$ be the pre-order word of $L_{i}$, and 
$\mathbf{w}_{i}$ be the permutation defined in Definition \ref{defn:weta}
with $Q=P_1$.
Then, we have $w(L_{i})=\mathbf{w}_{i}$ for $1\le i\le b$.
\end{prop}
\begin{proof}
Let $\mathbf{u}(q):=(u_1,\ldots,u_{an})$ be the step sequence for 
the Dyck path $q$.
Since we put labels on edges in the tree $T(q)$ by the post-order,  
the value $u_{j}$ is equal to the number of edges such that they have 
labels in $[1,j-1]$ and they are strictly left to the edge labeled by $j$.
An edge of the rooted tree for $q$ corresponds to a pair of 
$N$ and $E$ in the Dyck word $q$.
To visit edges by the pre-order is equivalent to visit 
$N$ steps in $q$ from left to right one-by-one.
From these observations, if we read the labels of $L_i$ by the pre-order, 
the word $w(L_i)$ is equal to $\mathbf{w}_{i}$.
\end{proof}

\begin{remark}
Proposition \ref{prop:wrt} is valid only for $Q=P_1$. 
This is because $\mathbf{v}_{i}$ is not non-decreasing in general for $Q=P_0$.
See Example \ref{exam:rDpStir}.
\end{remark}

We will construct a multi-permutation $\mu=(\mu_{1},\ldots,\mu_{abn})$ consisting 
of $b$ $i$'s with $1\le i\le an$ from the $b$ permutations $\mathbf{w}_{i}$ with $1\le i\le b$.

\begin{defn}
\label{defn:wdprime}
Suppose that a positive integer $i$ is uniquely written as 
$i=b(p-1)+q$ with $1\le p\le an$ and $1\le q\le b$.
We denote by $\mathbf{w'}_{i}:=(w'_{i,1},\ldots,w'_{i,an})$ with $1\le i\le b$ the inverse 
of the permutation $\mathbf{w}_{i}$.
Then, we define 
\begin{align}
\mu_{i}=w'_{q,p}.
\end{align}
\end{defn}

Let $Q\in\{P_{0},P_{1}\}$ be a rational Dyck path. 
Then, we have a $b$-Stirling permutation $\mu$ given a path $Q$.
Note that $\mu$ depends on the choice of $Q$ since $\mathfrak{u'}_{P}$ 
depends on $Q$.
\begin{theorem}
\label{thrm:distPQ}
Let $P\in\mathfrak{D}_{n}^{(a,b)}$ associated with $\mathfrak{u}_{P}$ and 
$\mu$ be a multi-permutation obtained from $\mathfrak{u}_{P}$. 
\begin{enumerate}
\item $\mu$ is a $b$-Stirling permutation.
\item $\mu=\zeta_{a}(\mathfrak{u'}_{P})$.
\end{enumerate}
\end{theorem}
\begin{proof}
We consider the case $Q=P_{1}$ since the case $Q=P_{0}$ is shown 
by a similar argument.
Below, we set $Q=P_1$. 

We prove the statement by induction on the Young order.
In case of $P=P_1$, 
all $\mathbf{v}_{i}=(0,\ldots,0)$.
From Definition \ref{defn:weta} and Definition \ref{defn:wdprime}, we have 
a multi-permutation $\mu_{P_1}$ given by 
\begin{align*}
\mu_{P_1}=(an^{b},(an-1)^{b},\ldots,1^{b}).
\end{align*}
It is obvious that $\mu_{P_1}$ is a $b$-Stirling permutation and 
$\mu=\zeta_{a}(0^{an})=\zeta_{a}(\mathfrak{u'}_{P_1})$.

We assume that the statement is true for all Dyck paths $P$ satisfying
$P'<_{Y}P$. 
We will show that the statement is also true for $P'$.
Let $\mathfrak{u}_{P}:=(u_1,\ldots,u_{an})$ and 
$\mathfrak{u}_{P'}:=(u'_1,\ldots,u'_{an})$ be step sequences associated 
with $P$ and $P'$.
By induction assumption, there exists a Dyck path $P$ such that 
$P'\lessdot_{Y}P$, $u_{t}=u'_{t}+1$ and $u_s=u'_{s}$ for $s\neq t$.	

To specify a Dyck path, we write $\mathbf{v}_{i}$ for $P$ 
as $\mathbf{v}_{i}(P)$.
The difference between $\mathbf{v}_{i}(P)$ and $\mathbf{v}_{i}(P')$ 
appears only at $t$-th entry, and other entries are the same by 
construction of $\mathbf{v}_{i}$.
Further, from definition of $\mathbf{v}_{i}(P)$, it is clear that 
$\mathbf{v}_{i}(P)$ is a non-decreasing sequence.

Let $\overline{\mathbf{v}(P)}:=(v_1(P),\ldots,v_{an}(P))$ be an integer 
sequence such that $v_{j}(P)$ is the $t$-th entry of $\mathbf{v}_{j}(P)$. 
Suppose that $au_{t}=bp+q$ is uniquely written by $0\le q\le b-1$.
Then, it follows from the definition of $\mathbf{v}_{i}$ that 
\begin{align}
\label{eqn:vp}
\begin{aligned}
&v_{k}(P)=p+1, \quad k\in[1,q], \\
&v_{k}(P)=p, \quad k\in[q+1,b-1].
\end{aligned}
\end{align}
Write $a(u_{t}+1)=bp'+q'$ with $0\le q'\le b-1$.
The sequence $\overline{\mathbf{v}(P')}$ is similarly given by 
Eqn. (\ref{eqn:vp}) by replacing $(p,q)$ by $(p',q')$.
Since $\mathbf{w'}_{i}(P)$ is the inverse permutation of $\eta(\mathbf{v}_{i}(P))$,
the value $p+1$ or $p$ in Eqn. (\ref{eqn:vp}) are the positions of $t$ 
in $\mathbf{w'}_{i}(P)$. 
Let $\mu^{t}$ be a multi permutation obtained from $\mu_{i}$ by deleting 
integers larger than $t$.
Then, the integer $t$'s are placed next to each other in $\mu^{t}$.
This implies that $\mu$ is a $b$-Stirling permutation. 

As in the previous paragraph, $p$ or $p+1$ for $P$ and $p'$ or $p'+1$
for $P'$ in Eqn. (\ref{eqn:vp}) indicate the positions of $t$ in $\mu^{t}$.
Note that the sum of $v_{i}(P)$ (resp. $v_{i}(P')$) is $au_{t}$ and $au_{t}+a$.
This means that the integer $t$'s in $\mu^{t}$ for $P'$ are 
placed right by $a$ compared to the case of $P$ by construction of $\mathbf{v}_{i}(P)$.
From this observation combined with the induction assumption, 
we have $\mu=\zeta_{a}(\mathfrak{u'}_{P'})$.
This completes the proof.
\end{proof}

\begin{cor}
For $Q=P_1$, we have $\mathbf{v}_{i}\lessdot_{Y}\mathbf{v}_{i+1}$ for all $i\in[1,b-1]$.
For $Q=P_{0}$, we have $\mathbf{v}_{i}-\mathbf{v}_{i+1}$ is a non-negative integer sequence 
for all $i\in[1,b-1]$.
\end{cor}
\begin{proof}
It is obvious from Eqn. (\ref{eqn:vp}) in Theorem \ref{thrm:distPQ}.
\end{proof}

\begin{remark}
In Theorem \ref{thrm:distPQ}, the multi-permutation $\mu$ roughly measures the distance 
from the fixed path $Q$ since it is given by $\zeta_{a}(\mathfrak{u'}_{P})$.
\end{remark}

\begin{example}
\label{exam:rDpStir}
Let $P$ be a rational Dyck path in $\mathfrak{D}_{2}^{(2,3)}$ with 
the step sequence $\mathfrak{u}_{P}=(0,1,1,3)$. 
\begin{enumerate}
\item $Q=P_0$.
The lowest path $P_{0}$ has the step sequence $\mathfrak{u}_{P_0}=(0,1,3,4)$.
We have 
\begin{align*}
\mathbf{v}_{1}=(0,0,2,1), \quad
\mathbf{v}_{2}=(0,0,1,1), \quad
\mathbf{v}_{3}=(0,0,1,0). 
\end{align*}
Thus the permutations $\mathbf{w}_{i}:=\eta(\mathbf{v}_{i})$ for 
$1\le i\le 3$ are given by 
\begin{align*}
\mathbf{w}_1=3142, \quad
\mathbf{w}_2=4132, \quad
\mathbf{w}_3=4231.
\end{align*}
By taking the inverses of $\mathbf{w}_{i}$ and constructing $\mu$, 
we obtain 
\begin{align*}
\mu=224442133311.
\end{align*}
Note that $\mu=\zeta_{2}((0,0,2,1))$.
\item $Q=P_1$. 
By a similar calculation to (1), we have 
\begin{align*}
\mathbf{v}_1=1423,\quad
\mathbf{v}_2=1423,\quad
\mathbf{v}_3=4213.
\end{align*} 
Thus $\mu=113332444221$, which is $\zeta_{2}((0,1,1,3))$.	
\end{enumerate}
\end{example}

\begin{theorem}
\label{thrm:stripdec}
Let $P\in\mathfrak{D}_{n}^{(a,b)}$ be a rational Dyck path and 
$\delta(P):=(q_1,\ldots,q_b)$ as in Definition \ref{defn:stripdec}. 
Let $\mathbf{i}$, $1\le i\le b$, be a $b$-tuple of Dyck words for $Q=P_{1}$.
Then, the step sequence of $q_i$ is $\mathbf{v}_{i}$.
\end{theorem}
\begin{proof}
The $b$-tuple of Dyck paths $\delta(P)$ are constructed from the height 
sequence $(h_1,\ldots,h_{bn})$ of $P$. 
In this process, we repeat $h_{i}$ $b$ times in the 
sequence $\mathfrak{H}$ (see Definition \ref{defn:stripdec}).
On the other hand, the sequences $\mathbf{v}_{i}$'s are constructed 
from the step sequence $\mathfrak{u}_{P}$ (see Eqn. (\ref{eqn:hsdecv1}) and 
(\ref{eqn:hsdecv2})).
In this process, we multiply $a$ on $\mathfrak{u}_{P}$.
This corresponds to repeating $h_{i}$ in $\mathfrak{H}$.

The observations above and Theorem \ref{thrm:distPQ} implies that
the Dyck paths $\delta(P)$ coincide with the Dyck paths $\mathbf{v}_{i}$,
$1\le i\le b$.
Thus, the step sequence of $q_{i}$ is $\mathbf{v}_{i}$.	 
\end{proof}

\subsection{Rotation for a \texorpdfstring{$b$-tuple}{b-tuple} of Dyck paths}
In this subsection, we show a rotation on a rational 
Dyck path by use of a $b$-tuple of Dyck paths.

From Theorem \ref{thrm:distPQ}, a rational Dyck path $P\in\mathfrak{D}_{n}^{(a,b)}$
corresponds to a $b$-Stirling permutation $\mu$.
We construct a parenthesis presentation of type $I$ from $\mu$ as in Section \ref{sec:pmDw},
and obtain the $b$-tuple of Dyck paths $\mathfrak{a}$ by the parenthesis 
presentation of type $I$ (see Definition \ref{defn:alphaast}).
Note that the step sequence of $\mu$ is obtained from the step sequence of $P$ 
by the action of $\zeta_{a}$.
In Section \ref{sec:rotmDw}, we define a rotation on $\mathfrak{a}$ for $(a,b)=(1,m)$.
Due to the action of $\zeta_a$, we modify the definition as follows.

Let $\epsilon(i)$ be the left parenthesis in $\alpha^{\ast}(\mu)$ as in Section 
\ref{sec:rotmDw}.
Then, We denote by $l(j)$, $1\le j\le a$, the label of the $j$-th $\ast$ 
left to $\epsilon(i)$.
We perform a rotation of size $n'$ at position $r$ on 
$a_{l(j)}$ for $1\le j\le a$, and the other $a_{p}$'s with $p\notin\{l(j) |1\le j\le a\}$
remains the same.

Then, the following proposition is a direct consequence of the definition of the rotation 
and Proposition \ref{prop:rotonmDw}.
\begin{prop}
The above mentioned operation on a $b$-tuple of Dyck paths constructed from $\zeta_a(\mu)$ 
by the parenthesis presentation of type $I$ gives a rotation on $P$.
\end{prop}

\begin{example}
Let $P$ be a rational Dyck path in $\mathfrak{D}_{2}^{(2,3)}$ 
as in Figure \ref{fig:rDyck}, and set $i=2$.
The step sequences for $b$ Dyck paths $\mathfrak{a}:=(q_1,q_2,q_3)$ 
are 
\begin{align*}
\mathfrak{u}(q_1)=(0,0,0,2), \quad
\mathfrak{u}(q_2)=(0,1,2,3), \quad
\mathfrak{u}(q_3)=(0,1,2,3). \quad
\end{align*}
The parenthesis presentation for $P$ is 
\begin{align*}
(**\underline{(}**(***)\underline{*}(***))*),
\end{align*}
where the second left parenthesis and its right-most $\ast$ are underlined.
From this, the size of the rotation is two.
Since $a=2$ and the first two $\ast$'s in the parenthesis presentation 
of type $I$ have labels $3$ and $2$, we perform 
a rotation on the second and the third Dyck paths in $\mathfrak{a}$.
Then, we obtain three step sequences $\mathfrak{u}(q'_{j})$:
\begin{align*}
\mathfrak{u}(q'_{1})=(0,0,0,2), \quad
\mathfrak{u}(q'_{2})=(0,0,1,3), \quad
\mathfrak{u}(q'_{3})=(0,0,1,3).
\end{align*}
Note that $\mathfrak{u}(q'_{i})$, $1\le i\le 3$, satisfy 
$\mathfrak{u}(q'_{3})\le \mathfrak{u}(q'_{2})\le\mathfrak{u}(q'_{1})$ in 
the Young order.
It is easy to see that triple Dyck paths $(q'_1,q'_2,q'_3)$ can be obtained from 
the Dyck path $P'$ whose step sequence is $(0,0,1,4)$.
The path $P'$ is obtained from $P$ by the rotation at position $2$.
Therefore, these two Dyck paths satisfy $P\lessdot_{\mathrm{rot}} P'$.
\end{example}

\section{Binary trees}
\label{sec:bt}
\subsection{binary trees}
A {\it complete binary tree} is a rooted tree such that each internal 
node has exactly two children as in Section \ref{sec:mptree}.
By abuse of notation, we call a rooted tree such that each node has at most 
two children a {\it binary tree}.
An edge of a binary tree is called a left edge or a right edge.
We denote by $\mathfrak{BT}_{n}$ the set of binary trees with $n$ edges,
or equivalently with $n+1$ nodes.

Following \cite{PreRatVie17}, we construct a word $\omega(T)$ on the alphabet $\{1,\bar{1},2,\bar{2}\}$ 
for a given binary tree $T\in\mathfrak{BT}_{n}$. 
We walk around a binary tree $B$ from the root by counterclockwise. 
When we walk on the left edge for the first time, we write a letter $1$, and 
when we walk on the left edge for the second time, we write a letter $\bar{1}$.
Similarly, when we walk on the right edge for the first (resp. second) time,
we write a letter $2$ (resp. $\bar{2}$).
We denote by $\omega(T)$ the word obtained as above.

We construct two words $\omega_1(T)$ and $\omega_2(T)$ from $\omega(T)$.
The word $\omega_1(T)$ is obtained from $\omega(T)$ by deleting $1$'s and $2$'s.
Similarly, the word $\omega_2(T)$ is obtained from $\omega(T)$ by deleting $1$'s and $\bar{2}$'s.
By replacing $\bar{1}$ by $N$ and $2$ or $\bar{2}$ by $E$, we obtain 
a path $\omega_1(T)$ and $\omega_2(T)$.

We will define a map 
$B:\mathfrak{D}_{n}^{(a,b)}\times\mathfrak{D}_{n}^{(a,b)}\rightarrow \mathfrak{BT}_{(a+b)n}$,
$(Q,P)\mapsto T$ with $P\le_{Y}Q$ in the Young order.
We denote $B(Q,P)$ for $Q=P$ by $B(P)$.
Suppose that $P$ is written in terms of an up step $N$ and a right step $E$ as 
\begin{align}
P=N^{i_1}E^{j_1}\ldots N^{i_r}E^{j_r}.
\end{align}
It is easy to define a binary tree for $N^{i}E^{j}$ as 
\begin{align}
N^{i}E^{j} \leftrightarrow
\tikzpic{-0.5}{[x=0.4cm,y=0.4cm]
\draw (1,-1)--(2.2,0.2)(3.4,1.4)--(4.6,2.6)node{$\bullet$}--(5.8,1.4)(7,0.2)--(8.2,-1);
\draw(2,0)node{$\bullet$}(3.6,1.6)node{$\bullet$}(5.6,1.6)node{$\bullet$}(7.2,0)node{$\bullet$};
\draw[dashed](2.2,0.2)--(3.4,1.4)(5.8,1.4)--(7,0.2);
\draw[decoration={brace,mirror,raise=5pt},decorate](4.6,2.6)--(1.1,-0.9);
\draw[decoration={brace,mirror,raise=5pt},decorate](8.1,-0.9)--(4.6,2.6);
\draw(7.5,2)node{$j$}(1.6,2)node{$i$};
}.
\end{align}
Note that the number of edges is $i+j$.
We define a binary tree for $P$ recursively as follows.
Suppose that $P'=N^{i_1}E^{j_1}\ldots N^{i_{r-1}}E^{j_{r-1}}$ and we have a 
corresponding binary tree $\mathrm{Tr}(P')$.
Then, the binary tree $\mathrm{Tr}(P)$ is obtained from $\mathrm{Tr}(P')$ by attaching 
the tree for $\mathrm{Tr}(P')$ at the leftmost leaf of $N^{i_r}E^{j_r}$.

The following lemma is an easy consequence of the construction of a binary tree 
$\mathrm{Tr}(P)$ for $P$.
Recall $\omega_{i}(\mathrm{Tr}(P))$ for $i=1,2$ are the words obtained from $\mathrm{Tr}(P)$. 
\begin{lemma}
\label{lemma:bt1}
Let $T:=\mathrm{Tr}(P)$ be a binary tree for $P$.
We have $\omega_{1}(T)=\omega_{2}(T)=P$.
\end{lemma}

We start from the binary tree $T:=\mathrm{Tr}(P)$ such that $\omega_{i}(T)=P$ for $i=1,2$, 
and construct a binary tree $B(Q,P)$ such that 
$\omega_{1}(B(Q,P))=Q$ and $\omega_{2}(B(Q,P))=P$.

Given two rational Dyck paths $P$ and $Q$ expressed in terms of words 
consisting of $N$ and $E$, we introduce a notion of 
{\it rotation} on $P$ to get $Q$.

Let $P=(p_1,\ldots,p_n)\in\{N,E\}^{n}$ and $Q=(q_1,\ldots,q_n)\in\{N,E\}^{n}$ be two 
rational Dyck paths, and $\mathrm{Tr}(P)$ be the corresponding binary tree for $P$.
By definition, we have $p_1=q_1=N$.
Let $k_0$ be the smallest integer such that 
\begin{align*}
&p_{l}=q_{l}, \quad \forall l\in[1,k_{0}-1], \\
&p_{k_{0}}\neq q_{k_{0}}.
\end{align*}
Since $p_1=q_1=N$ and $P<_{Y}Q$, we have $p_{k_{0}}=E$ and $q_{k_{0}}=N$.
Let $k_1$ and $k_2$ be a positive integer such that 
\begin{align*}
&p_{l}=E,\quad \forall l\in[k_{0},k_{1}-1], \\
&p_{k_{1}}=N,
\end{align*}
and 
\begin{align*}
&p_{l}=E,\quad \forall l\in[k_{1}+1,k_{2}-1], \\
&p_{k_{2}}=N.
\end{align*}
Let $e_{0}$ be the right edge in $\mathrm{Tr}(P)$ corresponding to $p_{k_{0}}=E$, 
which is $\overline{2}$ in the reading word $\omega_2(T)$.
Then, let $\mathrm{Tr}'(P')$ be the partial binary tree in $\mathrm{Tr}(P)$ corresponding to 
the partial path $P'=(p_{k_{1}},\ldots,p_{k_{2}-1})=NE\ldots E$.
We cut $\mathrm{Tr}(P)$ into four pieces. 
First we have $\mathrm{Tr}'(P')$. 
We denote by $B_{r}$ and $B_{t}$ the two pieces 
in $\mathrm{Tr}(P)\setminus \mathrm{Tr}'(P')$
such that $B_{r}$ contains the root of $\mathrm{Tr}(P)$.
We further divide into $B_{t}$ into two pieces.
Since the binary tree $B_{t}$ contains the right edge $e_{0}$, 
we have two binary trees $B^{(0)}_{t}$ and $B^{(1)}_{t}$ 
such that $B^{(0)}_{t}$ contains the edge $e_{0}$ and $B^{(1)}_{t}$
is the binary tree below $e_{0}$.

First, we glue the four binary trees $B_{r}$, $\mathrm{Tr}'(P')$, $B^{(0)}_{t}$ and 
$B^{(1)}_{t}$ into a binary tree $\mathrm{Tr}(P_1)$.
We glue the root of $B^{(0)}_t$ and the left-most leaf of $B_{r}$ 
into a binary tree and denote it by $B_{r}^{\mathrm{new}}$.
Secondly, we glue $B_{r}^{\mathrm{new}}$, $\mathrm{Tr}'(P')$ and $B_{t}^{(1)}$ 
into a binary tree such that we attach the right leaf of the 
edge $e_0$ in $B_{r}^{\mathrm{new}}$ and the root of $\mathrm{Tr}'(P')$, and attach the left-most leaf of 
$\mathrm{Tr}'(P')$ and the root of $B_{t}^{(1)}$.
We denote by $B(P_1,P)$ the newly obtained binary tree.
We call the operation to get $B(P_1,P)$ from $\mathrm{Tr}(P)$ {\it rotation} on
a binary tree.
Note that we have $P<_{Y}P_1<_{Y}Q$.
By successively applying rotations on $B(P)$, 
we obtain the binary tree $B(Q,P)$ from $B(P)$.

\begin{prop}
Given a binary tree $B(Q,P)$ constructed by successive rotations, we have 
two binary words $\omega_{1}(B(Q,P))=Q$ and $\omega_{2}(B(Q,P))=P$
from $B(Q,P)$.
\end{prop}
\begin{proof}
Let $T_1$ and $T_2$ be binary trees.
We write $T_1\rightarrow T_2$ if we concatenate two words 
$w_{i}(T_1)$ and $w_{i}(T_2)$ as $w_{i}(T_1)\circ w_{i}(T_2)$ 
for $i=1,2$.

We prove the statement by induction. 
When $Q=P$, we have $\omega_1(B(P,P))=P$ from Lemma \ref{lemma:bt1}.
We assume that $\omega_1(B(P_1,P))=P_1$ and take $P_2$ as a Dyck path 
such that $P_1\lessdot_{Y}P_2$. 
The difference of $P_1$ and $P_2$ is at position $k_0$.
We compare two words $\omega_1(B(P_1,P))$ and $\omega_1(B(P_2,P))$
such that $P_1\lessdot_{Y}P_2$.
Let $e_0$, $B_{r}$, $B_{t}^{(0)}$, $B_{t}^{(1)}$ and $\mathrm{Tr}'(P')$
be as above.
The word $w_{1}(B(P_1,P))$ is given by 
\begin{align*}
B_{t}^{(0)}\setminus\{e_0\}\rightarrow B_{t}^{(1)}\rightarrow e_{0}
\rightarrow \mathrm{Tr}'(P')\rightarrow B_{r}.
\end{align*}
Similarly, the word $w_{1}(B(P_2,P))$ is given by 
\begin{align*}
B_{t}^{(0)}\setminus\{e_0\}\rightarrow B_{t}^{(1)}
\rightarrow \mathrm{Tr}'(P')\rightarrow e_{0}\rightarrow B_{r}.
\end{align*}
Note that $\mathrm{Tr}'(P')=NE^{n}$ with a non-negative integer $n$.
Thus, we have $w_{1}(B(P_1,P))\lessdot_{Y}w_{1}(B(P_2,P))$.
By the choice of $e_0$, it is obvious that if $w_{1}(B(P_1,P))=P_1$,
then $w_{1}(B(P_2,P))=P_2$ by induction assumption.

By a similar observation as above, the words $\omega_2(B(P_{1},P))$ and 
$\omega_2(B(P_2,P))$ are given by 
\begin{align*}
B_{t}^{(0)}\rightarrow B_t^{(1)}\rightarrow \mathrm{Tr}'(P')
\rightarrow B_{r},
\end{align*}
which implies $\omega_2(B(P_{1},P))=\omega_2(B(P_2,P))=P$.
This completes the proof.
\end{proof}

\begin{example}
Let $P=NENENE^2NE^2$ and $Q=NEN^2E^2NE^3$ are two rational Dyck paths 
in $\mathfrak{D}_{2}^{(2,3)}$. 
We denote by $P_1=NEN^2E^3NE^2$ an intermediate path such that 
$P\lessdot_{Y}P_1\lessdot_{Y}Q$ in the Young order. 
The left picture of Figure \ref{fig:BT3} is the binary tree $\mathrm{Tr}(P)$.
By applying a rotation on $e_{0}$, we obtain the binary tree $B(P_1,P)$.
Further application of rotation on $e_{0}$ in $B(P_1,P)$, 
we obtain the binary tree $B(Q,P)$. 
\begin{figure}[ht]
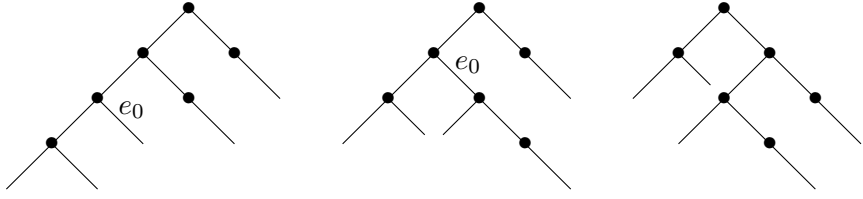

\tikzpic{-0.5}{[scale=0.6]
\draw(0,0)--(-4,-4)(0,0)--(2,-2)(-1,-1)--(1,-3)(-2,-2)--(-1,-3)(-3,-3)--(-2,-4);
\draw(0,0)node{$\bullet$}(-1,-1)node{$\bullet$}(-2,-2)node{$\bullet$}
(-3,-3)node{$\bullet$}(1,-1)node{$\bullet$}(0,-2)node{$\bullet$};
\draw(-1.25,-2.25)node{$e_{0}$};
}\quad
\tikzpic{-0.5}{[scale=0.6]
\draw(0,0)--(-3,-3)(0,0)--(2,-2)(-1,-1)--(2,-4)
(-2,-2)--(-1.2,-2.8)(0,-2)--(-0.8,-2.8);
\draw(0,0)node{$\bullet$}(-1,-1)node{$\bullet$}(-2,-2)node{$\bullet$}
(1,-3)node{$\bullet$}(1,-1)node{$\bullet$}(0,-2)node{$\bullet$};
\draw(-0.25,-1.25)node{$e_{0}$};
}\quad
\tikzpic{-0.5}{[scale=0.6]
\draw(0,0)--(-2,-2)(0,0)--(3,-3)(-1,-1)--(-0.3,-1.7)(1,-1)--(-1,-3)
(0,-2)--(2,-4);
\draw(0,0)node{$\bullet$}(-1,-1)node{$\bullet$}(2,-2)node{$\bullet$}
(1,-3)node{$\bullet$}(1,-1)node{$\bullet$}(0,-2)node{$\bullet$};
}
\caption{The binary trees $\mathrm{Tr}(P)$, $B(P_1,P)$ and $B(Q,P)$ from 
left to right.}
\label{fig:BT3}
\end{figure}
From the binary tree $T:=B(Q,P)$, we have
\begin{align*}
&\omega_{1}(T)=\overline{1211221222}=Q, \\
&\omega_{2}(T)=\overline{1}2\overline{1}2\overline{1}22\overline{1}22=P.
\end{align*}
\end{example}

\subsection{Horizontal and vertical strip-decomposition}
Let $P\in\mathfrak{D}_{n}^{(a,b)}$ be a rational Dyck path.
Dyck paths obtained by horizontal and vertical strip-decomposition
are directly obtained from $P$ as follows.
Suppose $P=N^{i_1}E^{j_1}\ldots N^{i_r}E^{j_r}$ by an integer $r$ 
and denote it by $P:=P[i_1,j_1;\ldots;i_r,j_r]$.
We define an {\it $(a,b)$-enlarged} rational Dyck path $\widetilde{P}$ 
by 
\begin{align}
\widetilde{P}=P[bi_1,aj_1;\ldots;bi_r,aj_r].
\end{align}
The Dyck path $\widetilde{P}$ is of size $2abn$, that is, $\mathfrak{D}_{2abn}^{(1,1)}$.
We will define an $r$-tuple of Dyck paths with an integer $r\in\{a,b\}$, which is denoted 
by $\widetilde{P^{(r)}}:=\{\widetilde{P^{(r)}_{i}}\}_{i=1}^{r}$ 
where each $\widetilde{P^{(r)}_{i}}$ is a Dyck path of length $2abn/r$.

\begin{defn}
Let $\widetilde{P}:=\overline{p(1)}\overline{p(2)}\ldots\overline{p(2abn)}\in\{N,E\}^{2abn}$ 
be a Dyck word presentation of the $(a,b)$-enlarged Dyck path. 
We define an $r$-tuple of Dyck paths $\overline{P^{(r)}}$ denoted as above by
\begin{align}
\widetilde{P^{(r)}_{i}}:=\overline{p(i)}\overline{p(r+i)}\ldots\overline{p(2abn-r+i)}.
\end{align}
\end{defn}

\begin{example}
We consider the path $NENEE\in\mathfrak{D}_{1}^{(2,3)}$.
The $(2,3)$-enlarged Dyck path $\widetilde{P}$ is given 
by $N^{3}E^{2}N^{3}E^{4}$.
Therefore, we have 
\begin{align*}
\widetilde{P^{(3)}}=(NENE,NENE,NNEE), \\
\widetilde{P^{(2)}}=(NNENEE,NENNEE).
\end{align*}
\end{example}

An $(a,b)$-enlarged Dyck path, and horizontal and vertical strip-decomposition 
of a path $P$ are related in the following way.
\begin{prop}
\label{prop:lbDw}
Let $P\in\mathfrak{D}_{n}^{(a,b)}$, $\delta$ and $\theta$ be the maps defined 
in Definition \ref{defn:stripdec} and \ref{defn:vsd}.
Then, we have 
\begin{align}
&\widetilde{P^{(b)}}=\delta(P), \\
&\widetilde{P^{(a)}}=\theta(P).
\end{align}
\end{prop}
\begin{proof}
Recall that we repeat the entries of the step sequence $\mathfrak{u}_{P}$ 
$a$ times to construct the vertical strip-decomposition $\theta(P)$.
This repeat of entries is realized by the enlargement of the Dyck path.
Therefore, we have  $\widetilde{P^{(a)}}=\theta(P)$.
One can show $\widetilde{P^{(b)}}=\delta(P)$ by a similar argument.
\end{proof}

\subsection{parenthesis presentation of type \texorpdfstring{$II$}{II} and binary tree}
In this subsection, we show that the parenthesis presentation of type $II$ is compatible
with a binary tree. This gives a decomposition of the binary tree into $b$ binary trees.

Let $P\in\mathfrak{D}_{n}^{(a,b)}$ and $P_0$ be the lowest path in $\mathfrak{D}_{n}^{(a,b)}$.
Let $\mathfrak{B}:=B(P,P_0)$ be the binary tree associated 
with the pair $(P,P_0)$.
We divide the left edges of $\mathfrak{B}$ into $b$ pieces,
and attach $(a-1)bn$ right edges at the right-most node. 
We denote by $\widetilde{\mathfrak{B}}$ the new binary tree.
We will construct two $b$-tuples of Dyck words 
$\mathfrak{d}^{(i)}:=(\mathfrak{d}^{(i)}_1,\ldots,\mathfrak{d}^{(i)}_{b})$ 
for $i=1,2$ from $\widetilde{B}$.
Recall we have two words $\omega_1(\widetilde{B}):=(\omega_{1}(1),\ldots,\omega_{1}(2abn))$ 
and $\omega_2(\widetilde{B}):=(\omega_{2}(1),\ldots,\omega_{2}(2abn))$.
We define $\mathfrak{d}^{i}$ by 
\begin{align}
\mathfrak{d}^{(i)}_{j}:=(\omega_{i}(b+1-j),\omega_{i}(2b+1-j),\ldots,\omega_{i}(2abn+1-j)),
\end{align}
where $i=1,2$ and $j\in[1,b]$.

\begin{prop}
Let $\mathfrak{d}^{(i)}$ with $i=1,2$ be two $b$-tuples of Dyck words defined as above.
Then, we have 
\begin{enumerate}
\item The word $\mathfrak{d}^{(1)}$ is equal to a word presentation of 
parenthesis presentation type II for $P$, namely, $\mathfrak{d}^{(1)}=\alpha^{II}(P)$. 
\item The word $\mathfrak{d}^{(2)}$ is equal to the lowest path $P_{0}$, namely, 
$\mathfrak{d}^{(2)}=\alpha^{II}(P_{0})$.
\end{enumerate}
\end{prop}
\begin{proof}
By combining the construction of $\alpha^{II}(\pi)$ 
in Section \ref{sec:mockDw}, Theorem \ref{thrm:distPQ} in Section \ref{sec:Hsdmp}, 
and Proposition \ref{prop:lbDw},
we have  $\mathfrak{d}^{(1)}=\alpha^{II}(P)$ 
and $\mathfrak{d}^{(2)}=\alpha^{II}(P_{0})$.
\end{proof}

\subsection{Duality between horizontal and vertical strip-decompositions}
In this subsection, we introduce the notion of duality for rational Dyck paths.
The duality between horizontal and vertical strip-decomposition plays a central 
role when we study the binary trees introduced in Section \ref{sec:bt}.
By the symmetry between $(a,b)$-Dyck paths and $(b,a)$-Dyck paths, we assume 
$a<b$ without loss of generality in this subsection.

Let $P$ be a rational Dyck path in $\mathfrak{D}_{n}^{(a,b)}$. 
When the path $P$ has a Dyck word presentation $P=p_1\ldots p_{(a+b)n}\in\{N,E\}^{(a+b)n}$,
we define the dual rational Dyck path $P^{\sharp}$ by 
\begin{align}
P^{\sharp}:=p^{\sharp}_{(a+b)n}\ldots p_1^{\sharp},
\end{align}
where $E^{\sharp}=N$ and $N^{\sharp}=E$.
The path $P^{\sharp}$ is in $\mathfrak{D}_{n}^{(b,a)}$.
Let $P_0$ be the lowest path below $P$ and $P_0^{\sharp}$ be its dual.

Given the binary tree $\mathfrak{B}:=B(P,P_{0})$, we cut 
left edges into $b$ pieces and right edges into $a$ pieces in $\mathfrak{B}$
and denote by $\widetilde{\mathfrak{B}}$ the new binary tree. 
The new binary tree $\widetilde{\mathfrak{B}}$ has $2abn$ edges in total.
We enumerate the edges in $\widetilde{\mathfrak{B}}$ by the post-order.
We construct $a$ binary trees $\mathfrak{b}_{i}$ with $1\le i\le a$ 
from $\widetilde{\mathfrak{B}}$ as follows.
Fix $i\in[1,a]$.
We delete all the edges except the edges labeled by $ak+1-i$ with $k\in[1,2bn]$.
We have $2bn$ remaining edges. 
We connect these edges keeping its relative positions in $\widetilde{\mathfrak{B}}$ 
and construct a binary tree $\mathfrak{b}_{i}$.

Similarly, given a binary tree $\mathfrak{C}:=B(P^{\sharp},P_0^{\sharp})$, 
we cut left edges into $a$ pieces and right edges into $b$ pieces and denote 
it by $\widetilde{\mathfrak{C}}$ the new binary tree.
We construct $a$ binary trees from $\widetilde{\mathfrak{C}}$ in the same manner 
as in the case of $\widetilde{\mathfrak{B}}$.
We denote by $\mathfrak{c}_{i}$, $1\le i\le a$, the $a$ binary trees 
obtained from $\widetilde{\mathfrak{C}}$.

\begin{theorem}
Let $\widetilde{\mathfrak{B}}$ and $\widetilde{\mathfrak{C}}$  
defined as above, and $\delta$ and $\theta$ the horizontal 
and vertical strip-decomposition. 
\begin{enumerate}
\item Let $\theta(P)=(\theta_1,\ldots,\theta_{a})$ and $\theta(P_{0}):=(\theta^{0}_1,\ldots,\theta^{0}_{a})$
be the vertical decomposition of $P$ and $P_{0}$ respectively.
We have $\mathfrak{b}_{i}=B(\theta_i,\theta^{0}_{i})$ for $1\le i\le a$.
\item 
Let $\mathfrak{u}(P^{\sharp})$ (resp. $\mathfrak{u}(P_{0}^{\sharp})$) be 
the step sequence of $P^{\sharp}$ (resp. $P_0^{\sharp}$), 
and $\pi$ (resp. $\pi_0$) be a multi-permutation associated with 
$\zeta_{b}(\mathfrak{u}(P^{\sharp}))$ (resp. $\zeta_{b}(\mathfrak{u}(P_{0}^{\sharp}))$).
Then, the Dyck words $\omega_1(\mathfrak{c}_{i})$ (resp. $\omega_2(\mathfrak{c}_{i})$) 
for $1\le i\le a$ coincide with the $a$-tuple of Dyck words $\alpha^{II}(\pi)$
(resp. $\alpha^{II}(\pi_0)$).
\item We have the duality between $\mathfrak{b}_{i}$ and $\mathfrak{c}_{i}$ 
for $1\le i\le a$:
\begin{align*}
&\omega_1(\mathfrak{b}_{i})=\omega_1(\mathfrak{c}_{a+1-i})^{\sharp}, \\
&\omega_2(\mathfrak{b}_{i})=\omega_2(\mathfrak{c}_{a+1-i})^{\sharp}.
\end{align*}
\end{enumerate}
\end{theorem}
\begin{proof}
(1)
The alphabets $\overline{1}$, and $\overline{2}$ or $2$ in 
the words $\omega_1(\mathfrak{B})$ and $\omega_2(\mathfrak{B})$ 
correspond to the left and right edges in $\mathfrak{B}$.
The construction of $\widetilde{\mathfrak{B}}$ from $\mathfrak{B}$
is equivalent to constructing an $(a,b)$-enlarged rational Dyck path
$\widetilde{P}$ from $P$.
From Proposition \ref{prop:lbDw}, we have 
$\mathfrak{b}_{i}=B(\theta_i,\theta^{0}_{i})$.

(2)
From Theorem \ref{thrm:distPQ}, Theorem \ref{thrm:stripdec} and 
Proposition \ref{prop:mock1} (apply for $\zeta_{b}(\mathfrak{u}(P^{\sharp}))$ and 
$\zeta_{b}(\mathfrak{u}(P_{0}^{\sharp}))$), it follows that 
the words $\omega_{1}(\mathfrak{c}_{i})$ and 
$\omega_{2}(\mathfrak{c}_{i})$ are given by 
the $a$-tuple of Dyck words of type $II$ $\alpha^{II}(\pi)$
and $\alpha^{II}(\pi_0)$.

(3)
The operation $\sharp$ reverses $N$ and $E$ to $E$ and $N$ 
in a rational Dyck path and also reverses the order of words.
By construction of $\widetilde{\mathfrak{B}}$ and $\widetilde{\mathfrak{C}}$, 
it is clear that the words constructed from the binary trees $\mathfrak{b}_{i}$
are obtained from the binary trees $\mathfrak{c}_{a+1-i}$ by the 
operation $\sharp$. This completes the proof.
\end{proof}

\begin{example}
We consider the same path $P=NENENE^2NE^2$ as Figure \ref{fig:rDyck} and the lowest
path $P_0=NENE^2NENE^2$.
Figure \ref{fig:btBT} is the binary trees $\mathfrak{B}(P,P_0)$ and 
$\widetilde{\mathfrak{B}}$.
The labels on $\widetilde{\mathfrak{B}}$ are expressed modulo $2$.
\begin{figure}[ht]
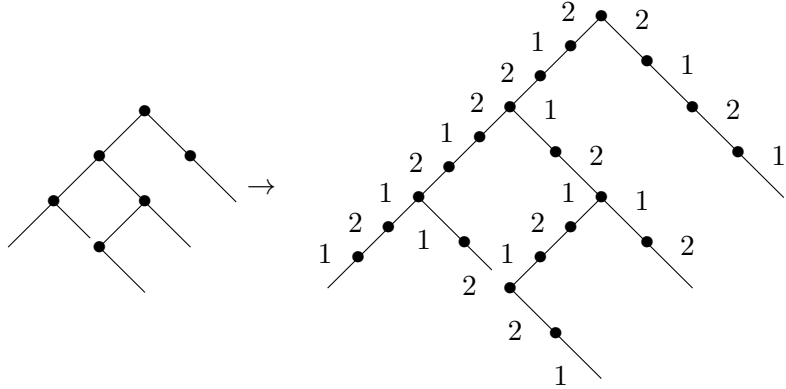

\tikzpic{-0.5}{[scale=0.6]
\draw(0,0)--(-3,-3)(0,0)--(2,-2)(-1,-1)--(1,-3)
(0,-2)--(-1,-3)--(0,-4)(-2,-2)--(-1.2,-2.8);
\draw(0,0)node{$\bullet$}(1,-1)node{$\bullet$}(-1,-1)node{$\bullet$}(-2,-2)node{$\bullet$}
(0,-2)node{$\bullet$}(-1,-3)node{$\bullet$};
}$\rightarrow$
\tikzpic{-0.5}{[scale=1.2]
\draw(0,0)--(-3,-3)(0,0)--(2,-2)(-1,-1)--(1,-3)
(0,-2)--(-1,-3)--(0,-4)(-2,-2)--(-1.2,-2.8);
\draw(0,0)node{$\bullet$}(1,-1)node{$\bullet$}(-1,-1)node{$\bullet$}(-2,-2)node{$\bullet$}
(0,-2)node{$\bullet$}(-1,-3)node{$\bullet$};
\draw(-1/3,-1/3)node{$\bullet$}(-2/3,-2/3)node{$\bullet$}(-4/3,-4/3)node{$\bullet$}
(-5/3,-5/3)node{$\bullet$}(-7/3,-7/3)node{$\bullet$}(-8/3,-8/3)node{$\bullet$}
(-1/3,-7/3)node{$\bullet$}(-2/3,-8/3)node{$\bullet$};
\draw(1/2,-1/2)node{$\bullet$}(3/2,-3/2)node{$\bullet$}
(-1/2,-3/2)node{$\bullet$}(-1/2,-7/2)node{$\bullet$}
(-3/2,-5/2)node{$\bullet$}(1/2,-5/2)node{$\bullet$};
\draw(-17/6,-17/6)node[anchor=south east]{$1$}(-15/6,-15/6)node[anchor=south east]{$2$}
(-13/6,-13/6)node[anchor=south east]{$1$}(-11/6,-11/6)node[anchor=south east]{$2$}
(-9/6,-9/6)node[anchor=south east]{$1$}(-7/6,-7/6)node[anchor=south east]{$2$}
(-5/6,-5/6)node[anchor=south east]{$2$}(-3/6,-3/6)node[anchor=south east]{$1$}
(-1/6,-1/6)node[anchor=south east]{$2$};
\draw(1/4,-1/4)node[anchor=south west]{$2$}(3/4,-3/4)node[anchor=south west]{$1$}
(5/4,-5/4)node[anchor=south west]{$2$}(7/4,-7/4)node[anchor=south west]{$1$};
\draw(-3/4,-5/4)node[anchor=south west]{$1$}(-1/4,-7/4)node[anchor=south west]{$2$}
(1/4,-9/4)node[anchor=south west]{$1$}(3/4,-11/4)node[anchor=south west]{$2$};
\draw(-7/4,-9/4)node[anchor=north east]{$1$}(-5/4,-11/4)node[anchor=north east]{$2$}
(-3/4,-13/4)node[anchor=north east]{$2$}(-1/4,-15/4)node[anchor=north east]{$1$};
\draw(-1/6,-13/6)node[anchor=south east]{$1$}(-3/6,-15/6)node[anchor=south east]{$2$}
(-5/6,-17/6)node[anchor=south east]{$1$};
}	
\caption{The left picture is the binary tree $\mathfrak{B}$ for the path $NENENE^{2}NE^{2}$. 
The right picture is $\widetilde{\mathfrak{B}}$ and its labels modulo $2$.}
\label{fig:btBT}
\end{figure}

By vertical strip-decomposition, we have two Dyck paths $\theta:=(\theta_1,\theta_2)$ and 
their canopies $\theta(P_{0})=(\theta_1^{0},\theta_2^{0})$
given by 
\begin{align*}
&\theta_{1}=N^{2}ENEN^{2}E^{2}NE^{2}, \quad \theta_{1}^{0}=N^{2}ENE^{2}N^{2}ENE^{2}, \\ 
&\theta_{2}=NEN^{2}ENE^{2}N^{2}E^{2}, \quad \theta_2^{0}=NEN^{2}E^{2}NEN^{2}E^{2}.
\end{align*}
The binary trees $\mathfrak{b}_{i}=B(\theta_i,\theta_{i}^{0})$ for $i=1,2$ are depicted 
as follows.

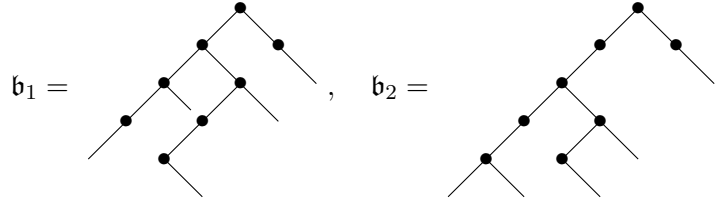
\begin{align*}
\mathfrak{b}_1=
\tikzpic{-0.5}{[scale=0.5]
\draw(0,0)--(-4,-4)(0,0)--(2,-2)(-1,-1)--(1,-3)(0,-2)--(-2,-4)--(-1,-5);
\draw(-2,-2)--(-1.3,-2.7);
\draw(0,0)node{$\bullet$}(1,-1)node{$\bullet$}(-1,-1)node{$\bullet$}
(-2,-2)node{$\bullet$}(-3,-3)node{$\bullet$}(0,-2)node{$\bullet$}
(-1,-3)node{$\bullet$}(-2,-4)node{$\bullet$};
},\quad
\mathfrak{b}_2=
\tikzpic{-0.5}{[scale=0.5]
\draw(0,0)--(2,-2)(0,0)--(-5,-5)(-2,-2)--(0,-4)(-1,-3)--(-2,-4)--(-1,-5)
(-4,-4)--(-3,-5);
\draw(0,0)node{$\bullet$}(1,-1)node{$\bullet$}(-1,-1)node{$\bullet$}
(-2,-2)node{$\bullet$}(-3,-3)node{$\bullet$}(-4,-4)node{$\bullet$}
(-1,-3)node{$\bullet$}(-2,-4)node{$\bullet$};
}
\end{align*}
Note that $\mathfrak{b}_{1}$ (resp. $\mathfrak{b}_{2}$) is obtained from $\widetilde{\mathfrak{B}}$ by taking 
the edges labeled by $1$ (resp. $2$). 

From the binary tree for $\mathfrak{C}=B(P^{\sharp},P^{\sharp}_{0})$, we obtain 
two binary trees:
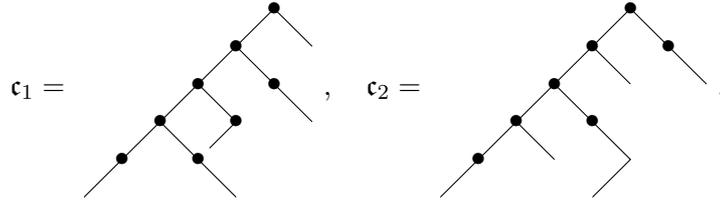
\begin{align*}
\mathfrak{c}_{1}=
\tikzpic{-0.5}{[scale=0.5]
\draw(0,0)--(1,-1)(0,0)--(-5,-5)(-1,-1)--(1,-3)(-2,-2)--(-1,-3)--(-1.7,-3.7)
(-3,-3)--(-1,-5);
\draw(0,0)node{$\bullet$}(-1,-1)node{$\bullet$}(-2,-2)node{$\bullet$}
(-3,-3)node{$\bullet$}(-4,-4)node{$\bullet$}(-1,-3)node{$\bullet$}(-2,-4)node{$\bullet$}
(0,-2)node{$\bullet$};
},\quad
\mathfrak{c}_{2}=
\tikzpic{-0.5}{[scale=0.5]
\draw(0,0)--(2,-2)(0,0)--(-5,-5)(-1,-1)--(0,-2)(-2,-2)--(0,-4)--(-1,-5)
(-3,-3)--(-2,-4);
\draw(0,0)node{$\bullet$}(1,-1)node{$\bullet$}(-1,-1)node{$\bullet$}
(-2,-2)node{$\bullet$}(-1,-3)node{$\bullet$}
(-3,-3)node{$\bullet$}(-4,-4)node{$\bullet$};
}.
\end{align*}
From $\mathfrak{c}_1$, we have two Dyck paths 
\begin{align*}
a_{1}^{1}=N^2E^2N^2ENE^2NE, \quad
a_{2}^{1}=N^2E^2NEN^2E^2NE, 
\end{align*}
Similarly, from $\mathfrak{c}_2$, we have two Dyck paths
\begin{align*}
a_{1}^{2}=N^2EN^2E^2NENE^2, \quad
a_{2}^{2}=N^2ENE^2N^2ENE^2.
\end{align*}

The step sequence for $P^{\sharp}$ (resp. $P_0^{\sharp}$) 
is $(0,0,1,1,2,3)$ (resp. $(0,0,1,2,2,3)$).
The parenthesis presentations are 
\begin{align*}
\alpha^{\ast}(\pi)=((**)*((**)*(**)*(**))*), \quad 
\alpha^{\ast}(\pi_{0})=((**)*(**)*((**)*(**)*)).
\end{align*}
From these parenthesis presentations, we have 
\begin{align*}
\alpha^{II}(\pi)=(a_{1}^{1},a_{1}^{2}), \quad
\alpha^{II}(\pi_{0})=(a_{2}^{1},a_{2}^{2}).
\end{align*}
We have the duality among four Dyck words:
\begin{align*}
\theta_1=(a_{1}^{2})^{\sharp}, \quad
\theta_2=(a_{1}^{1})^{\sharp}, \quad
\theta_1^{0}=(a_{2}^{2})^{\sharp}, \quad
\theta_2^{0}=(a_{2}^{1})^{\sharp}.
\end{align*}
\end{example}

\bibliographystyle{amsplainhyper} 
\bibliography{biblio}

\end{document}